\documentclass[a4paper,reqno]{amsart}

\usepackage{graphicx,pdfsync,amssymb, latexsym,amsfonts,amsbsy,color,subcaption,esint,mathtools}

\usepackage[foot]{amsaddr}

\usepackage[hidelinks]{hyperref}

\usepackage[top=1.1in, bottom=1in, left=1in, right=1in]{geometry}

\newcommand{\eprint}[1]{\href{https://arxiv.org/abs/#1}{arXiv:#1}}

\DeclareMathOperator{\Supp }{Supp}
\DeclareMathOperator{\Id }{Id} 
 
\DeclareMathOperator{\D}{div}
\DeclareMathOperator{\Tr}{Tr} 

\newtheorem{theorem}{Theorem}[section]
\newtheorem{lemma}[theorem]{Lemma}
\newtheorem{proposition}[theorem]{Proposition}
\newtheorem{definition}[theorem]{Definition}

\newtheorem{remark}[theorem]{Remark}

\def \TT  {\mathbb{T}}
\def \RR {\mathbb{R}}
\def \NN {\mathbb{N}}

\def \l {\lambda}

\def  \Rl { \overline{R}_l }
\def  \ul { \overline{u}_l }
\def  \pl { \overline{p}_l }

\def \X {\xi}

\newcommand{\comment}[1]{}

\numberwithin{equation}{section}

\begin{document}

\title[Stationary solutions and nonuniqueness of weak solutions]{Stationary solutions and nonuniqueness of weak solutions for the Navier-Stokes equations in high dimensions}

%
%
%
%
%
%
\author{Xiaoyutao Luo}
\address{Department of Mathematics, Statistics and Computer Science, University of Illinois At Chicago, Chicago, Illinois 60607}
\curraddr{Department of Mathematics, Statistics and Computer Science, University of Illinois At Chicago, Chicago, Illinois 60607}
\email{xluo24@uic.edu} 



\begin{abstract}
Consider the unforced incompressible homogeneous Navier-Stokes equations on the $d$-torus $\TT^d$ where $d\geq 4$ is the space dimension. It is shown that there exist nontrivial steady-state weak solutions $u\in L^{2}(\TT^d)$. The result implies the nonuniqueness of finite energy weak solutions for the Navier-Stokes equations in dimensions $d \geq 4$. And it also suggests that the uniqueness of forced stationary problem is likely to fail however smooth the given force is.
\end{abstract}

\maketitle

\section{Introduction}
The incompressible Navier-Stokes equations on the $d$-torus $\TT^d=\RR^d /\mathbb{Z}^d$ are the following system:
\begin{equation}\label{eq:NSE}\tag{NSE}
\begin{aligned}
\partial_t u - \nu \Delta u& +\D ( u\otimes u ) +\nabla p =0 \\
& \D u =0 ,
\end{aligned}
\end{equation}
where $u(x,t):\TT^d \times \RR^+ \to\RR^d$ is the unknown vector field, the scalar $p:\TT^d \to \RR$ is the pressure and $\nu$ is the viscosity. The system is supplemented by the periodic boundary condition $u(x+k,t)=u(x,t)$ for any $k\in \mathbb{Z}^d$. We study the nonuniqueness issue of the weak solutions of \eqref{eq:NSE} in dimensions $d\geq 4$.

The following weak formulation of \eqref{eq:NSE} will be used throughout the paper and the term ``weak solution'' refers to this definition in the sequel.
\begin{definition}\label{def:weaksolution}
A vector field $u \in C_w (0,T;L^2(\TT^d) )$ is a weak solution of \eqref{eq:NSE} if it satisfies:
\begin{enumerate}
\item $ u(t)$ has zero-mean on $\TT^d$ and is weakly divergence-free for all $t\in [0,T]$, namely for any $\phi(x) \in C^\infty_0(\TT^d  )$ 
\begin{equation}
\int_{\TT^d} \nabla \phi u(x,t) dx =0 \quad \text{for all $t\in [0,T]$.}
\end{equation}
\item For any test function $\varphi(x,t) \in C^\infty_0( [0,T] \times\TT^d  )$ such that $\varphi(x,t)$ is divergence-free in $x$ for all $t\in [0,T]$ we have
\begin{align*}
\int_\RR \int_{\TT^d} u \cdot ( \partial_t \varphi +  (u\cdot \nabla)\varphi + \nu  \Delta \varphi   ) dx dt =0.
\end{align*}
\end{enumerate}
\end{definition}

The system \eqref{eq:NSE} has been extensively studied by many researchers for a long time. In \cite{MR1555394} Leray constructed weak solutions $u\in L_t^\infty L^2\cap L^2_t H^1$ for divergence-free initial data $u_0 \in L^2(\RR^3)$. These solutions are termed Leray-Hopf weak solutions and also satisfy the energy inequality. A similar result was later obtained by Hopf in \cite{doi:10.1002/mana.3210040121} for the bounded domain with Dirichlet boundary condition. It is also worth noting that even though the global existence results of Leray and Hopf are best known for $d=2,3$, they can be carried over to $\RR^d$ or $\TT^d$ for $d\geq 4$ without much trouble. Let us recall the relevant facts in the following.
\begin{theorem}[Leray-Hopf weak solutions]
Let $\Omega=\TT^d$ or $\RR^d$ for $d\geq 3$. Given any divergence-free $u_0 \in L^2(\Omega)$ and $0< T\leq \infty$, there exists at least one weak solution $u\in C_w(0,T;L^2) \cap L^2(0,T; H^1)$ to \eqref{eq:NSE} and $u$ verifies additionally the energy inequality:
\begin{equation}\label{eq:energy_ineq}\tag{E.I.}
\|u(t) \|_2^2 +2\nu \int_{t_0}^t \|\nabla u(s) \|_2^2 ds \leq \|u(t_0) \|_2^2
\end{equation}
for any $t>0$ and a.e. $t_0 \in[0,t)$ including $0$.
\end{theorem}

Even though the existence of weak solutions has been know for quite some time, the questions of uniqueness and regularity of Leray-Hopf weak solutions in $d\geq 3$ remain unknown to date and are considered one of the most important issues in mathematical fluid mechanics. More specifically the following questions are still open.
\begin{enumerate}
\item \textbf{Global regularity:} For any divergence-free initial data $u_0 \in L^2(\Omega)$, does there exist a global Leray-Hopf solution $u \in C^\infty((0,\infty)\times \Omega )$\footnote{Here we mean that for any $k,m\in \NN$ $\partial_t^k \nabla^m_x u$ is bounded on $[\epsilon,\infty)\times \Omega $.  } ?
\item \textbf{Uniqueness:} For any divergence-free initial data $u_0 \in L^2(\Omega)$, is the Leray-Hopf weak solution $u(t)$ with initial data $u_0$ unique among all Leray-Hopf weak solutions with initial data $u_0$?
\end{enumerate}

One of the motivations of our work is to obtain some insights into the above question of uniqueness. We are not yet able to obtain nonuniqueness results for the Leray-Hopf weak solutions. Instead, we consider the following pathway:

\textbf{A possible approach:} 
Find the ``smoothest'' function space $X \subset C_w (0,T;L^2)$ so that there exist two weak solutions of \eqref{eq:NSE} $u, v \in X$ such that $u(0) \neq v(0)$.

Let 
$$
\mathcal{E}= \{u\in C_w (0,T;L^2  ): \|u \|_{L^\infty_t L^2} +   \|u \|_{L^2_t H^1} <\infty  \}.
$$ The nonuniqueness problem of Leray-Hopf weak solutions can then be formulated as finding a space $X $ so that $ X \subset  \mathcal{E} $ and weak solutions of \eqref{eq:NSE} in $X$ satisfy \eqref{eq:energy_ineq}.

Following such approach, we are able to obtain a nonuniqueness statement in dimensions $d \geq 4$ for the function space:
$$
X= \{u\in C_w (0,T;L^2 ):  \partial_t u=0 \quad \text{and $u \in H^{\beta}$ for $\beta<\frac{1}{200}$} \}.
$$ Detailed discussions on the main results will be given in Section \ref{subsection:mainresult}.

\subsection{Background and previous works}
In this subsection let us review the some of the related works for the Navier-Stokes equations in different settings. We first discuss the progress towards proving the global regularity and uniqueness for the 3D Navier-Stokes equations. Then a brief summary of nonuniqueness results and the method of convex integration is given. Lastly, we describe some existence and regularity results on the forced stationary problem in high dimensions.
\subsubsection{Regularity and uniqueness results in 3D}
Since the seminal work of Leray, there have been a substantial amount of conditional regularity and uniqueness results for \eqref{eq:NSE} with $d=3$, the most physical relevant case. Notably the classical Ladyzhenskaya-Prodi-Serrin criterion \cite{Ladayzhenskaya1967,MR0136885,Prodi1959} says that if additionally a Leray-Hopf weak solution also belongs to $L^q_t L^p_x$ for some $\frac{2}{q} + \frac{3}{p} \leq 1$ with $p>3$ then the solution is regular and unique among all Leray-Hopf solutions with the same initial data. The endpoint case $L^\infty_t L^3_x$ was solved by Escauriaza-Seregin-\v{S}ver\'{a}k in \cite{MR1992563} and extensive studies have been devoted to generalize and refine the classical Ladyzhenskaya-Prodi-Serrin criterion, see \cite{MR2237686,MR1876415,MR2854673,MR2564471,MR3465978,MR3475661} and reference therein for more conditional regularity and uniqueness results for the 3D NSE.

\subsubsection{Nonuniqueness results and the method of convex integration}
Very recently in \cite{1709.10033} Buckmaster and Vicol proved the nonuniqueness of weak solution for the 3D Navier-Stokes with finite energy using the method of convex integration. Originated from differential geometry, the method convex integration was successfully applied to the 3D Euler equations for the first time by De Lellis-Sz\' ekelyhidi Jr. in \cite{MR2600877}, where they obtained a bounded solution to the 3D Euler equations with compact support in space-time generalizing the results of Scheffer \cite{MR1231007} and Shnirelman \cite{MR1476315}. The pioneering work \cite{MR2600877} of De Lellis-Sz\' ekelyhidi Jr. was further developed in a series of papers \cite{MR3090182,MR3374958,MR3254331} and eventually, leads to the resolution of the famous Onsager's conjecture for the 3D Euler equations by Isett in \cite{1608.08301}. For a more detailed account of applications of convex integration in fluid dynamics, we refer to the survey \cite{MR3619726} by De Lellis and Sz\'ekelyhidi, Jr. and other interesting papers by different authors such as \cite{1610.00676,1701.08678,1708.05666,MR3488536,MR3479065,MR3633742}.

In the construction of Buckmaster-Vicol weak solutions are allowed to have any prescribed non-negative smooth functions as the energy profiles and hence $0$ is not the only weak solution with finite energy by taking a nontrivial compact energy profile. It is worth noting that the solutions constructed in \cite{1709.10033} are not known to be Leray-Hopf nor do they have finite dissipation $L^2_t H^1$. So these solutions do not obey the Leray structure theorem of the 3D Navier-Stokes equations on the interval of regularity for Leray-Hopf weak solutions, see for instance the original paper by Leray \cite{MR1555394} or the notes by Galdi \cite{MR1798753}. It would be very interesting to extend the results of Buckmaster-Vicol to the Leray-Hopf weak solutions.

Aside from using convex integration to construct weak solutions with ill behaviors, there is another pathway in pursuing the possible nonuniqueness of the Leray-Hopf weak solutions for the 3D NSE. As pointed out by Jia-\v{S}ver\'{a}k in \cite{MR3179576} one can also study the nonuniqueness issue via self-similar solution for $(-1)$-homogeneous initial data. In this direction Jia-\v{S}ver\'{a}k proved in \cite{MR3341963} that nonuniqueness of Leray-Hopf weak solutions in the regularity class $L^\infty_t L^3$ under a certain assumption for the linearized Navier-Stokes operator. Very recently Guillod-\v{S}ver\'{a}k showed numerically in \cite{1704.00560} that such assumption is likely to be true, though a rigorous mathematical proof of the assumption is not available so far. 

\subsubsection{The forced stationary problem}
Unlike the case when $d=3$, fewer results are available for the Navier-Stokes equations in high dimensions $d \geq 4$. We mention here a few studies on the forced stationary problem of NSE with the presence of external force as in a sense they are closely related to the main results in this paper. The forced stationary problem consists of the following equations:
\begin{equation}\label{eq:fNSE}
\begin{cases}
 - \nu \Delta u +\D ( u\otimes u ) +\nabla p =f &\\
 \D u =0 &
\end{cases}\qquad \text{for all $x\in \Omega$}.
\end{equation}

\begin{remark}
In parallel with Definition \ref{def:weaksolution} for the unforced NSE, we search for the weak solution to the above system that is $L^2 (\Omega)$ and verifies \eqref{eq:fNSE} in the sense of distribution for all divergence-free test functions $\varphi \in C^\infty_c(\Omega)$ (or $C^\infty_0(\TT^d)$ if $\Omega =\TT^d$). Different types of formulations of the problem \eqref{eq:fNSE} have been considered in the literature, cf. \cite{MR1284206,MR2506072,MR2107439}.
\end{remark}

The existence of regular solutions to \eqref{eq:fNSE} has been known under various assumptions on the force $f$ and the domain $\Omega$. In the seventies, Gerhardt studied the four-dimensional case in \cite{MR520820}, where he proved that if $f\in L^p$ then if there exists a solution, then $u\in W^{2,p}$. Since then there have been a considerable amount of studies on the forced stationary problem in high dimensions. Frehse-R\r u\v zi\v cka \cite{MR1343646} and Struwe \cite{MR1343647} showed the existence and regularity of the solutions in five dimensions. Later Frehse-R\r u\v zi\v cka obtained existence of regular solutions in bounded six-dimensional domain in \cite{MR1276763} and on torus in dimensions up to $d=15$ in \cite{MR1373209}. Recently, Farwig and Sohr \cite{MR2486616} considered the general $d$-dimensional case where in particular a uniqueness result was obtained for small force $f$. We refer readers to \cite{MR2506072,MR2107439,MR1284206} and reference therein for more interesting results on the forced stationary problem. 

Despite the existence result on the regular solutions for the forced stationary case, the question of uniqueness of regular solutions to \eqref{eq:fNSE} remains mostly open, particularly when the data $f$ is large. We offer here a partial result in this direction showing that at least weak solutions are not unique to the stationary problem \eqref{eq:fNSE} when $f=0$.

\subsection{Main results}\label{subsection:mainresult}
Before giving the main results of the paper, let us state some of the motivations of this work. In the hope of better understanding of the nonuniqueness issue of the Navier-Stokes system in both forced stationary case and unforced time-dependent case, we study the following Liouville-type problem:

\textbf{(Q)} Consider \eqref{eq:NSE} for $\Omega=\RR^d$ or $\TT^d$. Does there exist nontrivial stationary weak solution $u$,  i.e. $\partial_t u=0$ so that $u\in  L^p(\Omega)$ for some $2\leq p < d$ (or $H^s(\Omega)$ for some $0 \leq s < \frac{d-2}{2}$)?

From an energy balancing point of view, one can think of \textbf{(Q)} as investigating how strong the nonlinear term is in producing nontrivial energy flux to balance linear dissipation. Such phenomenon is closely related to the Onsager's conjecture and the concept of anomalous dissipation for the Navier-Stokes equations \cite{MR1302409,MR2665030,MR2422377}. If \textbf{(Q)} has a positive answer, then it might be possible to use this mechanism of nontrivial energy flux to construct Leray-Hopf solutions satisfying \eqref{eq:energy_ineq} with strict inequality. In terms of $L^q_t L^p_x$ norms, for the Navier-Stokes equations in dimension $d \geq 3$ the scaling of the conditions implying uniqueness \cite{Ladayzhenskaya1967,MR0136885,Prodi1959,MR1876415} corresponds to $\frac{2}{q} + \frac{d}{p}=1$ while the one implying energy equality is $\frac{2}{q} + \frac{2}{p}=1$ \cite{MR0435629}.  So in view of such scaling gap, finding a Leray-Hopf solution with strict energy inequality could be the first step towards obtaining the nonuniqueness. We plan to address these issues in our future studies.

From a uniqueness point of view, a positive answer to the above question \textbf{(Q)} immediately would imply the nonuniqueness of weak solutions of \eqref{eq:NSE} in the class $L^p$. Indeed, as initial data, $u$ gives rise to two different weak solutions: a stationary solution $u(t)=u$ itself and the other one $v(t)$ is Leray-Hopf. Then $v$ as a Leary-Hopf weak solution satisfies the energy inequality but $u$ does not and hence they are different. Moreover such existence result would also imply the nonuniqueness of weak solutions of the forced stationary problem \eqref{eq:fNSE} for a particular force: there exists a force $f=0 \in C^\infty_0(\TT^d)$ such that the system \eqref{eq:fNSE} admits two different solutions, with one trivial solution being regular and the other nontrivial one in $L^p$. It would be very interesting obtain the same result for other nontrivial forces.

The main results of this note are the following theorems.
\begin{theorem}[Existence of stationary weak solutions]\label{thm:l2_stationary}
Suppose $d \geq 4$. There exists nontrivial steady-state weak solution $u\in  L^{2}(\TT^d)$ of \eqref{eq:NSE}.
\end{theorem}

\begin{remark}
In fact, we proved a slightly stronger result that the solution lies in $H^{\beta}(\TT^d)$ for every $\beta < \frac{1}{200}$.
\end{remark}

As discussed in the paragraph preceding the statement of our main results, nonuniqueness of weak solutions of \eqref{eq:NSE} and of the stationary problem \eqref{eq:fNSE} in $d\geq 4$ both follow from Theorem \ref{thm:l2_stationary}.
\begin{theorem}[Nonuniqueness of the NSE in $d\geq 4$]\label{thm:nonuniqueness}
Suppose $d \geq 4$. There exists divergence-free initial data $u_0 \in  L^{2}(\TT^d)$ so that $u_0$ admits at least two different weak solutions of \eqref{eq:NSE} in the sense of Definition \ref{def:weaksolution}.
\end{theorem}

\begin{theorem}[Nonuniqueness of the stationary problem]\label{thm:nonuniqueness_fNSE}
Suppose $d \geq 4$. There exists a force $f \in C^\infty_0(\TT^d)$ so that the system \eqref{eq:fNSE} admits at least two different weak solutions of \eqref{eq:fNSE} denoted as $u $ and $v$ so that $u=0\in C^\infty_0(\TT^d)$ is trivial while $v\in L^2(\TT^d)$ is nontrivial.
\end{theorem}

To the author's knowledge, Theorem \ref{thm:l2_stationary} is the first result showing the existence nontrivial stationary weak solutions for the system \eqref{eq:NSE}. So it provides a positive answer to \textbf{(Q)} for $p=2$ (or any $p>2$ sufficiently close to 2). It is interesting that even without the presence of external force, the nonlinear term itself can produce enough energy flux to balance the linear dissipation. 

We remark that even though adaptations of the convex integration scheme have already been used for other partial differential equations in fluid dynamics, it is the first time that such method is used for the Navier-Stokes equations in dimensions $d \geq 4$. Moreover, it is also worth noting that the scheme used by Buckmaster-Vicol does not generate stationary weak solutions for the 3D NSE even if one takes a constant energy profile. The reason is that the building blocks of their construction are time-dependent by default. It is not clear whether one can adapt their scheme to obtain stationary weak solutions in 3D NSE. The existence of nontrivial stationary weak solutions of the 3D NSE still remains open. Another benefit of considering stationary weak solutions is that without the time-dependence of the solution our proof is much more streamlined.

It also appears that our current scheme is compatible with the time-dependent case, and it is likely that we can also obtain weak solutions with any given energy profile as in \cite{MR3090182,1701.08678,1709.10033}. However, it is unlikely that one is able to obtain nonuniqueness of Leray-Hopf weak solutions using current techniques without incorporating substantially new ideas.

\subsection{Comments on the role of intermittency}\label{subsection:comments_intermittency}
Compared with other inviscid or ipodissipative models, such as the Euler equations, the Muskat problem, the Surface Geostrophic equations or the ipodissipative Navier-Stokes equations \cite{1610.00676,MR3479065,1709.05155,1708.05666} where results of nonuniqueness-type have been obtained, the Navier-Stokes system has a dissipation term $\nu\Delta u$ with two derivatives, making it much harder to find suitable building blocks in the convex integration scheme. To resolve this issue one has to start with building blocks that are intermittent, which means that Bernstein's inequality is highly saturated. The concept of intermittency is crucial to the theory of turbulence in hydrodynamics \cite{MR1428905} and it is both instructive and interesting to understand its role from a mathematical point of view, see \cite{MR0495674,MR2566571,MR3152734,1802.05785}. Besides the work of Buckmaster-Vicol \cite{1709.10033}, the idea of using building blocks that are intermittent has been used for other systems with diffusions, such as the transport-diffusion equations \cite{1806.09145,1712.03867}.

Let us briefly discuss the concept of intermittency using the Littlewood-Paley decomposition as follows. Suppose $u_q = \Delta_q u$ is a Littlewood-Paley projection at frequency of size $\sim 2^q$, then the intermittency $D \in[0,d]$ of $u$ is measured by
\begin{equation}
\|u_q \|_\infty \sim 2^{q \frac{d-D}{2} }  \| u_q \|_2 \quad  \text{for all $q$ sufficiently large}.
\end{equation}
When $D =d$ the function $u$ is homogeneous spatially and all $L^p$ norms are of the same order. When $D =0$ the function $u$ has extreme intermittency and Bernstein's inequality is fully saturated. In view of the behavior in physical space, $D$ roughly measures the concentration level of $u$ in the sense that $u$ is concentrated on some set of dimension $D$. The Beltrami flows used in \cite{MR3090182,MR3374958,MR3254331,MR3530360} for the 3D Euler equations and the Mikado flows used in the resolution of the Onsager's conjecture \cite{1608.08301} all have intermittency $D=3$. In contrast, the intermittent Beltrami flows constructed by Buckmaster-Vicol \cite{1709.10033} can be made to have intermittency $D=0$.\footnote{To get $D=0$ one has to change the parameters there for the intermittent Beltrami flows. And for their final perturbation $0 < D <1$.} 

Heuristically speaking if a weak solution of \eqref{eq:NSE} is of intermittency $D \geq d-2$ then it is regular. So it appears impossible to have stationary weak solutions of \ref{eq:NSE} with intermittency bigger than $d-2$. This is due to the fact that linear term dominates the nonlinear term in this regime and the problem becomes subcritical. We show this heuristics by considering the energy flux through each Littlewood-Paley shell as follows. Assuming only local interactions between scales, i.e. $\D (u \otimes u) \cdot u_q \sim \D (u_q \otimes u_q) \cdot u_q$ for simplicity. Then consider the energy flux equation obtained by multiplying \eqref{eq:NSE} with $\Delta_q u_{q}$ and then integrating in space:
\begin{align*}
\frac{d}{dt}\|u_q \|_2^2 +\text{Linear term}=\text{Nonlinear term}
\end{align*}
with
\begin{align*}
\text{Linear term} =  (\Delta u_q,u_q )  \sim 2^{2q} \|u_q \|_2^2  ,
\end{align*}
and
\begin{align*}
\text{Nonlinear term} =   ( \D (u \otimes u) \cdot u_q) \sim (\D (u_q \otimes u_q) \cdot u_q )\lesssim \| u_q \|_2^2 \| u_q \|_\infty \sim 2^{q \frac{d-D}{2} }  \| u_q \|_2^3.
\end{align*}
Due to the fact that $\| u_q\|_2 \to 0$ as $q \to \infty$ we have $\text{Linear term} \geq \text{Nonlinear term}$ when $D \geq d-2 $. 

Finally we remark that the building blocks that we are using have intermittency dimension $1$, which is the limitation forcing us to work in $d\geq 4 $. More discussions on the intermittency and its role in the construction can be found in Section \ref{subsection:perturbation}.

\subsection{Some remarks on notations}
Throughout the manuscript we use the following standard notations. 
\begin{itemize}
\item $\|\cdot \|_p := \| \cdot \|_{L^p(\TT^d)}$ is the Lebesgue norm for any $1\leq p \leq \infty$ and $\|\cdot \|_{C^m} := \sum_{0\leq  i\leq m}\| \nabla^i \cdot \|_{\infty}$ for any $m$  is the H\"older norm on $d$-torus $\TT^d$.

\item We say a function $f(x):\TT^d \to \RR$ (or $\TT^d \to \RR^d$ for the vector case) is smooth if $f$ has continuous derivative or any order and we denote $f\in C^\infty(\TT^d)$. The space $C_0^\infty(\TT^d) $ consists of all smooth functions with zero-mean.
\item $x \lesssim y$ stands for the bound $x \leq C y$ with some constant $C$ which is independent of $x$ and $y$ but may change from line to line. Then $x \sim y$ means $x \lesssim y$ and $y \lesssim x$ at the same time. We use $x\ll y$ to indicate $x\leq cy$ for some small constant $0<c<1$.

\item The space $C^\infty_0(\TT^d)$ is the set of smooth functions with zero-mean on $\TT^d$. $\fint_{\TT^d }= \frac{1}{|\TT|^d} \int_{\TT^d }$ is the average integral and for any function $f\in L^1(\TT^d)$, its average is denoted by $\overline{f}=\fint_{\TT^d }f$.

\item For any $\TT^d$-periodic function $f\in L^p(\TT^d)$ and $\sigma >0$, the notation $f(\sigma \cdot)$ is the scaled $\sigma^{-1} \TT^d$-periodic function $f(\sigma x)$ so that $\|f(\sigma \cdot) \|_p= \|f  \|_p   $ for $L^p$ norms. 

\item For vectors $a,b \in \RR^d$, $a\otimes b$ is the matrix with $(a\otimes b)_{ij} = a_i b_j$ and $a \mathring{ \otimes } b =  a_i b_j (1 -\delta_{ij} ) $ is the trace-less product. For matrix-value functions $f=f_{ij}$ and $g=g_{ij}$, $\D f = \partial_i f_{ij}$ and $f:g =f_{ij} g_{ij}$.
 
\item $\Delta_q$ is the standard periodic Littlewood-Paley projections on to the dyadic frequency shell $ 2^{q-1}\leq |\X| \leq 2^{q+1}$ for any $q \geq -1$ and $\Delta_{\leq q} =\sum_{r \leq q} \Delta_r$ and $\Delta_{\geq q} =\sum_{r \geq q} \Delta_r$.

\item We also use wavenumber projections to simplify notations. For any $\l\in\NN$ define $\mathbb{P}_{\leq \l} =\sum_{q:2^q \leq \l} \Delta_q$ and $\mathbb{P}_{\geq \l} =\Id-\mathbb{P}_{\leq \l}$.

\item $\mathcal{S}^{n \times n}_+$ denotes the set of positive definite symmetric $n \times n$
matrices and $Q^d =[0,1]^d$ is the $d$-dimensional box.
\end{itemize}

\subsection{Plan of the paper}
The rest of the paper is organized as follows. In Section \ref{Section:outline} we give the main idea of the construction. The main proposition of the paper and how we design the perturbation $w_n$ in the iteration process are discussed. We then show that Theorem \ref{thm:l2_stationary} follows from the main proposition. Next, the building blocks that we called concentrated Mikado flows are introduced in Section \ref{Section:mikado}, where a key technical tool in the form of a commutator-type estimate is also presented. With all ingredients in hand, we prove the main proposition in Section \ref{Section:proof}.

\subsection*{Acknowledgments}
The author is grateful to his advisor Alexey Cheskidov for many stimulating discussions and for valuable comments after reading earlier versions of the manuscript. The author is partially supported by the NSF grant DMS 1517583 through his advisor Alexey Cheskidov.

\section{Outline of the construction}\label{Section:outline}
In this section, we briefly introduce the main idea of the construction. The proof of Theorem \ref{thm:l2_stationary} is based on an iteration scheme. We construct by induction a sequence of smooth solutions to the Navier-Stokes-Reynolds system verifying a certain set of estimates which guarantees the convergence to a stationary weak solution to \eqref{eq:NSE} in $L^2(\TT^d)$. The iteration process is then summarized in Proposition \ref{prop:main_iteration}. After stating the main proposition, we give a proof of Theorem \ref{thm:l2_stationary}. Lastly, we outline the explicit form of the velocity perturbation and explain its important role in the induction process.

\subsection{Navier-Stokes-Reynolds system}\label{subsection:NSRsystem}
Let us first recall the Navier-Stokes-Reynolds system\footnote{We normalize $\nu$ to $1$ without loss of generality.} in the time-dependent case:
\begin{equation}\label{eq:time_NSRE}
\begin{cases} 
\partial_t u -\Delta u  + \D (u  \otimes u  )  +\nabla p  =\D R  &\\
\D u  =0. &
\end{cases}
\end{equation}
where $R$ is a trace-less symmetric matrix usually termed Reynolds stress in the literature. 

The system \eqref{eq:NSRE} arises naturally in the study of weak solutions of the 3D Navier-Stokes equations and the 3D Euler equations. The tensor $R$ measures the distance to the \eqref{eq:NSE}. In fact every weak solution to the original Navier-Stokes equations \eqref{eq:NSE} can generate a family of solutions $v_l$ of \eqref{eq:time_NSRE}. Let $ u $ be a weak solution of \eqref{eq:NSE} and define
$$
v_l= u*\eta_l
$$ 
where $*\eta_l$ is some kind of averaging process in space (for example frequency localization to wavenumber $\lesssim l^{-1}$ or standard smoothing mollifier at length scale $\sim l$) , then $v_l$ is a solution to 
\begin{equation}
\begin{cases}
\partial_t v_l-\Delta v_l + \D (v_l \otimes v_l )  +\nabla p_l =\D R_l&\\
\D v_l =0. &
\end{cases}
\end{equation}
for some suitable pressure $p_l$ where the symmetric trace-less matrix $R_l$ is defined by
\begin{align*}
R_l=  (u  \mathring{ \otimes } u)*\eta_l  -( u*\eta_l) \mathring{ \otimes }  (u*\eta_l ). 
\end{align*}

Since we are constructing stationary weak solutions to \eqref{eq:NSE}, it is convenient to consider the following stationary Navier-Stoke-Reynolds system:
\begin{equation}\label{eq:NSRE}\tag{NSR}
\begin{cases} 
-\Delta u  + \D (u  \otimes u  )  +\nabla p  =\D R  &\\
\D u  =0. &
\end{cases}
\end{equation}
For our consideration here all $ u $, $p $ and $R $ are assumed to be $C_0^\infty(\TT^d)$, i.e. smooth and zero-mean.

A sequence of solution triplet $\{ (u_n,p_n,R_n) \}_{n\geq 1}$ to \eqref{eq:NSRE} will be constructed in the proof, and we measure the solutions $(u_n,p_n,R_n)$ by two parameters, a frequency $\l_n$ and a amplitude $\delta_{n}$. These two parameters are explicitly defined by
\begin{equation}\label{eq:def_for_l_n_d_n}
\begin{aligned}
&\l_n =  \left \lceil{a^{b^n}}\right \rceil \\
&\delta_n = \l_n^{-2\beta}
\end{aligned}
\end{equation}
where $ \left \lceil{x}\right \rceil$ denotes the smallest integer $n \geq x$, the parameters $a>0$ and the exponential frequency gap $b >1$ are large depending on $\beta $, which is the $L^2$ regularity of the solution $u$. The double exponential growth of $\l_n$ is critical for our proof, see Proposition \ref{prop:inversediv_commutator2} in Section \ref{Section:mikado}.

Starting with zero solution $(u_1,p_1,R_1 )=(0,0,0)$ we will construct a sequence of solutions $(u_n,p_n,R_n )$ of the system \eqref{eq:NSRE} so that the following set of estimates is verified:
\begin{align}
\| R_n\|_1 & \leq \delta_{n+1}\l_n^{-2 \alpha} \label{eq:hypothesis1} \tag{H1}\\
\| u_n\|_2 & \leq 1-\delta_{n }^{1/2} \label{eq:hypothesis2} \tag{H2}\\
\| \nabla u_n\|_2 & \leq \l_n \delta_{n }^{1/2  } \tag{H3}\label{eq:hypothesis3}
\end{align}
where $0 < a<\beta$ is another small parameter depending on $\beta$ and $b$. The exact values of all the parameters will be given in Section \ref{Section:proof}.

We remark that unlike the schemes used for the 3D Euler equations, cf. \cite{1701.08678,1608.08301,MR3090182}, here the Reynolds stress $R_n$ is measured in $L^1$ norm rather than $L^\infty$ norm. The reason is that $L^d$ is the critical norm for \eqref{eq:NSE} in $d$ dimension. So as pointed out in the introduction, no stationary solution exists in $L^\infty(\TT^d)$\footnote{Here and in what follows, weak solutions refer to Definition \ref{def:weaksolution}.} regardless of the dimensions and hence $R_n$ can not have any decay in $L^\infty$.

With these in mind we state the main proposition.
\begin{proposition}\label{prop:main_iteration}
There exists a sufficiently small $0<\beta \ll 1$ such that we can find $b>1$, $0<\alpha \ll\beta$  and $a \gg 1$ so that there exists a sequence of smooth solution triplets $(u_n,p_n,R_n )$ to the system \eqref{eq:NSRE} for $n \in \NN$ starting from $(u_1,p_1,R_1 )=(0,0,0)$ verifying \eqref{eq:hypothesis1}, \eqref{eq:hypothesis2} and \eqref{eq:hypothesis3}. Moreover each velocity increment $u_n -u_{n-1}$ is nontrivial and we have the estimate:
\begin{align}\label{eq:main_iteration_w_n}
\|u_{n} -u_{n-1}  \|_2 + \frac{1}{\l_n}\| \nabla u_{n} -\nabla u_{n-1} \|_2 \leq   \l_n^{-\beta}.
\end{align}
\end{proposition}
\begin{remark}
For example we can take $\beta=\frac{1}{200}$, $b=5$, $\alpha=10^{-6}$ independent of dimension $d$ and $a$ sufficiently large depending on some implicit constants from the computation. Such choice of $\beta$ and $b$ is definitely not optimal. By optimizing one can take larger $\beta$ as the dimension $d$ increases. However one is unlikely to get close to the critical space $H^{\frac{d-2}{2}}$ or $L^d$ without substantially new ideas. So we do not pursue additional improvement in the regularity using the current scheme in this direction.
\end{remark}
\begin{proof}[Proof of Theorem \ref{thm:l2_stationary}]
Let $(u_n,p_n,R_n)$ be the sequence obtained from Proposition \ref{prop:main_iteration} and let $0<\beta'<\beta$ First, by the Sobolev interpolation 
\begin{align*}
\sum_{k \geq n}\|u_{k} -u_{k-1} \|_{H^{\beta'}(\TT^d)} &\leq \sum_{k \geq n} \|u_{k} -u_{k-1} \|_{H^{1}(\TT^d)}^{\beta'} \|u_{k} -u_{k-1} \|_{L^{2}(\TT^d)}^{1-\beta'}\\
&\lesssim \sum_{k \geq n} \|\nabla u_{k} - \nabla u_{k-1} \|_2^{\beta'} \|u_{k} -u_{k-1} \|_2^{1-\beta'} .
\end{align*}
Then directly from the estimate \eqref{eq:main_iteration_w_n} we find that
\begin{align*}
\sum_{k \geq n}\|u_{k} -u_{k-1} \|_{H^{\beta'}(\TT^d)} &\lesssim \sum_{k \geq n} ( \l_n \delta_{n })^{\beta'}  \\
& \lesssim \sum_{k \geq n}  \l_n^{\beta' -\beta} \lesssim \l_{k}^{-\beta'}
\end{align*}
which means $u_{n}$ is uniformly bounded in $H^{\beta'}$. So by the compactness of the embedding $H^{\beta'}(\TT^d) \hookrightarrow L^{2}(\TT^d)$ after possibly passing to a subsequence which we still denote as $u_{n}$, we find that there exists a $u\in L^2(\TT)$ so that 
\begin{align*}
u_{n} \to u \quad \text{strongly in $L^2$}.
\end{align*}
Now we need to show that $u$ is a weak solution of \eqref{eq:NSE}. This is done by a standard argument. Let $\varphi(x) \in C^\infty_0(\TT^d)$. Multiplying \eqref{eq:NSRE} by $\varphi$ and integrating in space give
\begin{align*}
\int_{\TT^d} -\varphi \cdot \Delta u_n + \varphi \cdot  \D (u_n \otimes u_n )  +\varphi \cdot \nabla p_n = \int_{\TT^d} \varphi \cdot  \D R_n.
\end{align*}
Using the fact that $u_n$ is divergence-free and integrating by parts we find that
\begin{align*}
\int_{\TT^d}  u_n\cdot  \Delta\varphi +  \int_{\TT^d} u_n\cdot (u_n\cdot\nabla) \varphi       - \int_{\TT^d} \nabla \varphi : R_n=0.
\end{align*}
Due to the strong convergence of $u_n$ in $L^2$ the first two terms converge to their natural limit:
\begin{align*}
& \Big| \int_{\TT^d}  u_n\cdot  \Delta\varphi - \int_{\TT^d}  u \cdot  \Delta\varphi\Big| \leq  \|u_n-u \|_2 \|\Delta\varphi \|_2 \to 0  \quad \text{as $n \to \infty$};
\end{align*}
and
\begin{align*}
 \Big| \int_{\TT^d} u_n\cdot (u_n\cdot\nabla) \varphi - \int_{\TT^d} u \cdot(u \cdot\nabla) \varphi \Big| & \leq \Big| \int_{\TT^d} (u_n-u)\cdot (u_n\cdot\nabla) \varphi -   u \cdot((u-u_n) \cdot\nabla) \varphi \Big| \\
& \leq \|u-u_n \|_2\|u_n \|_2 \| \nabla  \varphi \|_\infty  +\|u-u_n \|_2 \|u \|_2 \| \nabla  \varphi \|_\infty  \\
& \to 0 \quad \text{as $n \to \infty$}.
\end{align*}
By the estimate \eqref{eq:hypothesis1} it follows that
$$
R_n \to 0  \quad \text{strongly in $L^1$},
$$
and then
\begin{align*}
& \Big|\int_{\TT^d} \nabla \varphi : R_n \Big|\leq \|R_n \|_1 \|\nabla \varphi \|_\infty  \to 0   \quad \text{as $n \to \infty$} .
\end{align*}
So $u \in H^{\beta'}$ for some $\beta'>0$ and verifies the weak formulation of \eqref{eq:NSE}. To recover the pressure $p$ associated to the solution $u$ we can use the formula
\begin{align*}
p_n = \Delta^{-1} \D \D (u_n \otimes u_n +  R_n ).
\end{align*}
To ensure $p_n$ converges to some $p$ in $L^r$ for some $r>1$ we need better convergence of $u_n$ and $R_n$. This can be done by using the fact that $u_n \to u$ in $H^{\beta'}$ to obtain that there is some small $\epsilon>0$ so that after possible relabeling
\begin{align*}
&u_n \to u \quad \text{strongly in $L^{2+2\epsilon}$}\\
&R_n \to 0 \quad \text{strongly in $L^{1+\epsilon}$}.
\end{align*}
Then by the $L^p$ boundedness of the Riesz transform for $p>1$ we know
\begin{align*}
p_n \to p \quad \text{strongly in $L^{1+\epsilon}$}
\end{align*}
for some $p \in L^{1+\epsilon}$.
\end{proof}

\subsection{The perturbation \texorpdfstring{$w_n$ }{wn }}\label{subsection:perturbation}
The main task is to construct $w_n$ given $(u_{n-1},p_{n-1},R_{n-1})$. The exact definition for the precise $w_n$ will be given later in Section \ref{Section:proof}. It should be noted that the exact scheme is more complicated than what we describe here.

We aim to design the $w_n$ so that it gives rise to a new solution triplet $(u_{n},p_n,R_n)$ verifying \eqref{eq:hypothesis1}, \eqref{eq:hypothesis2} and \eqref{eq:hypothesis3}. To the leading order $w_n$ will be of the form 
$$
w_n(x)=\sum_{i,k} a_{i,k}(R_{n-1}) \psi_{i,k}^{\mu_n}(\sigma_n x)
$$
where $a_{i,k}$ are the coefficients for the concentrated Mikado flow that will be defined by $R_{n-1}$ and have low frequency $\sim \l_{n-1}$, the variable $\mu_n$ is concentration parameter so that each $\psi_{i,k}^{\mu_n}(x)$ is supported in some cylinder with radius $\mu_n$ and $\sigma_n$ is oscillation parameter so that $\psi_{i,k}^{\mu_n}(\sigma_n \cdot)$ is $\sigma_n^{-1} \TT^d$-periodic and supported on cylinders of radius $\l_n^{-1}$ on $\TT^d$. Thus $w_n$ has frequency $\l_n$ in the sense that $\l_n =\sigma_n \mu_n$. 

To explain the role of $\sigma_{n}$ and $\mu_n$ let us recall that for the new Reynolds stress $R_n$ we need to solve the divergence equation:
\begin{align*}
\D R_n = &\underbrace{\D \Delta w_n}_{\text{Linear error}} + \underbrace{\D ( w_{n}\otimes w_n +R_{n-1} )}_{\text{Oscillation error}}\\
&+ \underbrace{\D ( u_{n-1}\otimes w_n+  w_n\otimes  u_{n-1} )}_{\text{Quadratic error}} +\nabla ( p_n-p_{n-1}).
\end{align*}

The idea of convex integration is to use the interaction $w_n\otimes w_n$ to balance the previous Reynolds stress $R_n$ in the sense that
\begin{equation}
\D ( w_{n}\otimes w_n +R_n ) + \nabla p_n = \text{High frequency error term}.
\end{equation}
And more importantly we need to do this while keeping the ``Linear error'' under control. So it is required for $w_n$ to have small intermittent dimension $D < d-2$. On the other hand in order to control the ``Oscillation error'' we need to have large spacing between each Fourier mode of $w_n$ which is parametrized by $\sigma_{n}$.

The role of $\sigma_n$ is to ensure $|\psi_{i,k}^{\mu_n}(\sigma_n \cdot)|^2$ only has Fourier modes of multiples of $\sigma_n $, namely
\begin{align*}
|\psi_{i,k}^{\mu_n}(\sigma_n \cdot)|^2 =\sum_{m\in \mathbb{Z}^d } \widehat{|\psi_{i,k}^{\mu_n}|^2}(m) e^{2\pi \sigma m \cdot x}
\end{align*}
such that the ``High frequency error term'' obeys the right inductive estimate. To this end we will invoke a commutator-type estimate that takes advantage of the fast oscillation of $|\psi_{i,k}^{\mu_n}(\sigma_n \cdot)|^2$, for which it is required that
\begin{align}\label{eq:l_n_sigma}
\l_{n-1} \ll \sigma_{n}.
\end{align}

On the other hand since $\psi_{i,k}^{\mu_n}$ is designed to be supported on small cylinders of radius $\mu_n$, one expects the saturation of the Bernstein inequality up to the exponent $d-1$:
\begin{equation}
\| \psi_{i,k}^{\mu_n} \|_p \sim \mu_n^{\frac{d-1}{2} -\frac{d-1}{p}}  
\end{equation}  
namely $ \psi_{i,k}^{\mu_n}$ is of intermittency dimension $D=1$. This is also the reason our construction only works in dimensions $d\geq 4$ since for the 3D Navier-Stokes equations the solution is regular if the intermittency is equal of greater than $3-2=1$. And as we shall see in the following discussion, the intermittent dimension of the perturbation $w_n$ is strictly bigger than $1$.

Taking the fast oscillation parameter $\sigma_{n}$ into account the ``Linear error'' verifies
\begin{equation}
\begin{cases}
\|\nabla w_n \|_1 \lesssim \delta_{n+1}^\frac{1}{2} \mu_n^{-\frac{d-1}{2}} \l_n &\quad \\
 \|\nabla w_n \|_2 \lesssim \delta_{n+1}^\frac{1}{2}   \l_n&
\end{cases}\text{with $\l_n =\sigma_{n} \mu_n$}.
\end{equation}
To see what is the intermittency dimension of $w_n$, let $ \mu_n=\l_n^{\mu}$ and $\sigma_n=\l_n^{\sigma}$ for some $\mu+\sigma =1$. Then by simple algebra and the definition of intermittency in term of Littlewood-Paley decomposition (cf. Section \ref{subsection:comments_intermittency}) we find that
$$
d-D= \mu (d-1  ).
$$
Since $\mu<1$ due to the requirement that $\l_{n-1} \ll \sigma_n$, we can infer that $D>1$, which is the reason that the construction breaks down in 3D. Furthermore, to make sure the ``Linear error'' is small in $L^1$ we need 
\begin{align}\label{eq:mu_n_sigma}
\sigma_{n} \ll \mu_{n}
\end{align}
which will ensure that $D < d-2$ when $d\geq 4$, cf. Section \ref{subsection:comments_intermittency}.
Then \eqref{eq:mu_n_sigma} and \eqref{eq:l_n_sigma} together imply that $\l_{n-1} \ll \l_{n}$, namely the frequency gap $b\gg 1$. In view of the quadratic relation
\begin{align*}
\|w_n \otimes w_n \|_1 \sim \| R_{n-1}\|_1,
\end{align*}
the regularity of $w_n$ is determined by $R_{n-1}$ so the large gap $b\gg 1$ results in a very small amount of regularity of $w_n$.

\section{Concentrated Mikado flows}\label{Section:mikado}
In this section, the building blocks of the solution sequence are constructed. Based on a variation of the Mikado flows introduced by Daneri and  Sz\'ekelyhidi Jr. in \cite{MR3614753}, we called these building blocks concentrated Mikado flows. As pointed out in the introduction, the idea of increasing the concentration of the flows is not new. Very recently, we learned that Modena and Sz\' ekelyhidi Jr. had used a similar idea to tackle the nonuniqueness problem for transport equation and continuity equation, see \cite{1712.03867,1806.09145}. However, being a systems of equations, the Navier-Stokes equations are fundamentally different than the transport equation and the continuity equation, which are both scalar equations. Here we point out some of the major differences between our construction and the ``Mikado densities'' and ``Mikado fields'' constructed by  Modena-Sz\' ekelyhidi Jr. in \cite{1712.03867,1806.09145}.
\begin{enumerate}
\item The error terms in \cite{1712.03867,1806.09145} (counterparts of the Reynolds stress in our setting, see Section \ref{subsection:NSRsystem}) are not matrices but vectors, therefore one does not need the geometric lemma for the space of symmetric traceless matrices, i.e. Lemma \ref{lemma:geometric}.
\item To resort the ``leakage'' when partitioning the Reynolds stress, we use another index $i$ so that there are multiple flows in the same direction but at different locations.
\item The estimates for the perturbations in \cite{1712.03867,1806.09145} are scaling invariant, while in our case scalings are much more complicated since we will take advantage of the superexponetial nature of \eqref{eq:def_for_l_n_d_n} using a commutator estimate, Proposition \ref{prop:inversediv_commutator2}.
\end{enumerate}

\subsection{The velocity profile \texorpdfstring{$\psi_{i,k}^\mu$}{psi}}\label{subsection:profile_function}
We first choose the velocity profile of the flow in this subsection. We fix a profile function $\psi: \RR  \to \RR  $ with $\Supp \psi \subset [ 1/2,1]$ so that we have 
\begin{align}\label{eq:profile_low_frequency}
\int_{\RR}  \psi\, dx =0 .
\end{align}

Let $\mathbb{K} \subset \mathbb{Z}^d$ be a given finite set of lattice vectors and $  N \in \NN $. Since $d \geq 4$ we can then choose a collection of points $p_{ i,k}$ for any $k\in \mathbb{K}$, $ 0\leq i\leq N $ and a number $\mu_0 >0$ with the following properties:
Let 
$$
l_{i,k} := \{p_{i,k}+t k+m:t\in \RR , m\in  \mathbb{Z}^d \} \subset \TT^d
$$ 
be the $\TT^d$-periodization of the line passing through $p_{i,k}$ in the $k$ direction. Since $\TT^d = \RR^d/\mathbb{Z}^d$ and $k \in \mathbb{K}  \subset \mathbb{Z}^d$, the line only goes around the box $Q^d=[0,1]^d$ finitely many times. So if we let 
$$
N_{\delta}(l_{i,k}) :=\{x+h: x\in l_{i,k},\, |h|\leq \delta \}
$$ 
denote the closed $\delta$-neighborhood of $l_{i,k}$, then
\begin{align*}
N_{ \mu_0^{-1}}(l_{i,k}) \cap N_{\mu_0^{-1}}(l_{j,k'})=\emptyset \quad \text{if  }  k\neq k' \,\, \text{or} \,\, i\neq j .
\end{align*}

For each direction $k \in \mathbb{K}$, any $0\leq i \leq N$ and any $\mu \geq \mu_0$ we define a profile function $\psi_{i,k}^\mu$ as 
\begin{align}\label{eq:def_psi}
\psi_{i,k}^\mu(x) =c_{k,p}   \mu^\frac{d-1}{2} \psi ( \mu \text{dist}(x,l_{i,k}))
\end{align}
where $c_{k,p} $ are normalizing constants so that 
\begin{align}\label{eq:def_psi_l_2}
\fint_{\TT^d } |\psi_{i,k}^\mu(x) |^2 dx =1 .
\end{align}
It is easy to see that due to \eqref{eq:profile_low_frequency}, we have
\begin{align}\label{eq:def_psi_osc}
\int_{\TT^{d} }  \psi_{i,k}^\mu(x) dx= 0  .
\end{align}
Indeed, to show \eqref{eq:def_psi_osc} one can use cylindrical coordinates along the direction $k$ as $(z,r,\theta_1,\dots,\theta_{d-2})$. Then 
\begin{align*}
\int_{\TT^d }   \psi ( \mu\text{dist}(x,l_{i,k})) dx
&=\int \prod_{j} f_j(\theta_1,\dots ,\theta_{d-2})^{\alpha_j}  \psi (  \mu r) dz dr d\theta_1 \dots d\theta_{d-2} \\
& =0.
\end{align*}

\subsection{Definition of concentrated Mikado flows}
It is clear that $\nabla \psi_{i,k}^\mu \cdot k =0$ since $\psi_{i,k}^\mu$ is a smooth function on $\TT^d $ whose level sets are concentric periodic cylinders with axis $l_k$.  Immediately we have the following properties.
\begin{enumerate}
\item Every $\psi_{i,k}^\mu k$ is divergence free: $\nabla \cdot \psi_{i,k}^\mu k =0$.
\item $\psi_{i,k}^\mu k$ solves $d$-dimensional Euler equations: $\D (\psi_{i,k}^\mu k \otimes \psi_{i,k}^\mu k )= \psi_{i,k}^\mu k \cdot \nabla \psi_{i,k}^\mu k=0$ .
\item These vector fields $\psi_{i,k}^\mu k$ have disjoint support: $\Supp \psi_{i,k}^\mu k \cap \Supp \psi_{j,k'}^\l =\emptyset$ if $k\neq k'$  or $i\neq j$.
\item $\psi_{i,k}^\mu k$ has intermittency dimension $D=1$, namely
\begin{align}
\| \nabla^m \psi_{i,k}^\mu  \|_p \lesssim \mu^{m} \mu^{  \frac{d-1}{2}  -\frac{d-1}{p}}.
\end{align}
\end{enumerate}

Recall that $\mathcal{S}^{d \times d}_+$ is the set of positive definite symmetric $n \times n$. The next geometric lemma allows us to form any $R$ in a compact subset of  $\mathcal{S}^{d\times d}_+$. A proof can be found in \cite{MR3340997} or \cite{MR0065993}.
\begin{lemma}\label{lemma:geometric}
For any compact subset $\mathcal{N} \subset\subset \mathcal{S}^{d\times d}_+$, there exists $\l_0 \geq 1$ and smooth functions $\Gamma_k \in C^\infty(\mathcal{N};[0,1] )$ for any $k \in \mathbb{Z}^d$ with $|k| \leq \l_0$ such that 
\begin{align*}
R = \sum_{k \in \mathbb{Z}^d, |k| \leq \l_0} \Gamma_k^2(R) k\otimes k \qquad \text{for all } \, R \in \mathcal{N} .
\end{align*}
\end{lemma}

Finally we can define the concentrated Mikado flows $W_i^{\mu}(R,x):\mathcal{B}\times \TT^d \to \RR^d$ as follows. We first apply Lemma \ref{lemma:geometric} with $\mathcal{N}= \mathcal{B}$, where
$$
\mathcal{B}=\{R \in \mathcal{S}^{d\times d}_+: |\Id -R| \leq \frac{1}{2}  \}
$$ 
is the ball of radius $1/2$ centering at $\Id $ in $\mathcal{S}^{d\times d}_+$. This fixes the direction set $\mathbb{K} =\{k \in \mathbb{Z}^d: |k| \leq \l_0\} $, where $\l_0$ is obtained from Lemma \ref{lemma:geometric}. Then given $N \in \NN$, by Section \ref{subsection:profile_function} there is a $\mu_0 >1$ depending on $N$, $d$ and $\mathbb{K}$ so that for any $\mu \geq \mu_0$ we let 
\begin{equation}
W_i^{\mu}(R,x) = \sum_{k \in \mathbb{Z}^d } \Gamma_k (R)\psi_{i,k}^\mu k \quad \text{for any $R \in \mathcal{B}$ and $0\leq i \leq N$}.
\end{equation}

\begin{remark}
The first lower index $i$ in $\psi_{i,k}^\mu k$ is to have multiple flows in the same direction $k$. So each flow $W_i^{\mu}(R,x)  $ can only interact with itself to recover the matrix $R$ as shown in the lemma below. This ensures the proper separation properties and is needed to control the ``leakage'' when partitioning the Reynolds stress in Section \ref{Section:proof}. 
\end{remark}

Let us summarize the properties in the lemma below for future reference. 
\begin{lemma}\label{lemma:mikado_R}
Suppose $d \geq 4$ and let $\mathcal{B}$ be the ball of radius $1/2$ centering at $\Id $ in $\mathcal{S}^{d\times d}_+$. For any $N \in \NN $ there exists $\mu_0 >1$ depending on $N$ and $d$ so that the divergence-free smooth vector fields $W_i^{\mu}(R,x): \mathcal{B}\times \TT^d \to \RR^d$ defined above for $\mu \geq \mu_0$ and indexed by $0 \leq i \leq N$ have the following properties.
\begin{align}
&W_i^{\mu}(R,x)\otimes W_i^{\mu}(R,x) = R + \sum_{k\in \mathbb{Z}^d} \Gamma^p_k(R)     \big[\big(\psi_{i,k}^\mu \big)^2 -  1 \big]~ k\otimes k \quad \text{for all}\,\, R \in \mathcal{N} .\\
&\Supp W_i^{\mu}(R,x) \cap \Supp W_j(R,x) =\emptyset \quad \text{if } \,i\neq j.
\end{align}
Moreover the profile function $\psi_{i,k}^\mu$ obeys the bound:
\begin{align}
\| \nabla^m \psi_{i,k}^\mu  \|_p \lesssim \mu^{m} \mu^{  \frac{d-1}{2}  -\frac{d-1}{p}} \quad \text{for any $1\leq p\leq \infty$ and $m\in \NN$}.
\end{align}
where the implicit constant depends on $i$, $k$, $p$ and $m$ but is dependent of $\mu$.
\end{lemma}

\subsection{A commutator estimate}
Next, we need a commutator-type estimate involving functions with fast oscillation, which is crucial in obtaining the $L^2$ decay of the perturbation $w_n$. 
It should be noted that a similar result for $p=1$ and $2$ has been established in \cite{1709.10033} using a different method. However, our result here requires a weaker assumption.

\begin{proposition}[Commutator for fast oscillation]\label{prop:inversediv_commutator2}
For any small $\theta >0$ and any large $N >0$ there exists $M\in \NN$ and $\l_0 \in \NN$ so that for any $\mu ,\sigma \in \NN$ satisfying $\l_0 \leq \mu \leq \sigma^{1-\theta} $ the following holds.
Suppose $a \in C^\infty(\TT^d) $ and let $C_a >0$ be such that 
\begin{equation*}    
\|\nabla^i a \|_\infty \leq C_a \mu^i \quad \text{for any $0 \leq i\leq M$}.
\end{equation*}
Then for any $\sigma^{-1}\TT^d$ periodic function $f \in L^p(\TT^d)$, $1< p <\infty$, the following estimates are satisfied.
\begin{itemize}
\item If $p \geq 2$ is even, then 
\begin{equation}\label{eq:holder_commutator1}
\big\|   a  f   \big\|_p \lesssim_{p,d,\theta,N} \|a \|_p    \|     f\|_p + C_a \|f \|_p \sigma^{-N}.
\end{equation}
\item If $\fint_{\TT^d } f =0 $ then for $0\leq  s \leq 1$:
\begin{equation}\label{eq:inversediv_commutator2}
\big\| |\nabla|^{-1} ( a  f)  \big\|_{p} \lesssim_{p,s,d,\theta,N}   \sigma^{-1+s} \big\||\nabla|^{-s}( a  f) \big\|_p+ C_a \|f \|_p \sigma^{-N}.
\end{equation}
\end{itemize}
All the implicit constants appeared in the statement are independent of $a$, $\mu$ and $\sigma$.
\end{proposition}

\begin{remark}
In fact, \eqref{eq:inversediv_commutator2} also holds for other $-1$ degree homogeneous Fourier multipliers, for example the inverse divergence operator $\mathcal{R}$ defined by \eqref{eq:def_operatorR} in Section \ref{Section:proof}.
\end{remark}

The proof of Proposition \ref{prop:inversediv_commutator2} is not difficult and we give a one using the Littlewood-Paley decomposition in the Appendix \ref{appendix:proofofproposition}. 

The significance of Proposition \ref{prop:inversediv_commutator2} is clearer upon recalling the Ansatz for the velocity increment: 
\begin{eqnarray}\label{eq:wn_ansatz}
w_n(x)=\sum_{i,k} a_{i,k}(R_{n-1}) \psi_{i,k}^{\mu_n}(\sigma_n x).
\end{eqnarray}
So by the usual H\"older's inequality we can get the trivial estimate
\begin{align*}
\|w_n \|_2 \leq \sum_{i,k} \big\| a_{i,k}(R_{n-1}) \big\|_\infty \big\| \psi_{i,k}^{\mu_n}(\sigma_n \cdot) \big\|_2
\end{align*}
which is too big for the final solution $u$ to be in $L^2$.

In contrast, if taking into account the fast oscillation $\l_{n-1 } \ll \sigma_{n}$ we would apply Proposition \ref{prop:inversediv_commutator2}, and then \eqref{eq:wn_ansatz} can be estimated as
\begin{align*}
\|w_n \|_2 \lesssim \big\| a_{i,k}(R_{n-1}) \big\|_2 \big\| \psi_{i,k}^{\mu_n}(\sigma_n \cdot) \big\|_2  ,
\end{align*}
and since 
$$ 
\big\| a_{i,k}(R_{n-1}) \big\|_2 \ll \big\| a_{i,k}(R_{n-1}) \big\|_\infty
$$
we shall see this approach indeed gives the desired bound for final solution $u$ to be in $L^2$.

\section{Proof of Proposition \ref{prop:main_iteration}}\label{Section:proof}
\newcounter{steps}[section]
Let us give the main steps of the proof here. We prove by induction. It is obvious that $(0,0,0)$ verifies all the inductive estimates \eqref{eq:hypothesis1}, \eqref{eq:hypothesis2} and \eqref{eq:hypothesis3}. Given a solution triplet $(u_{n-1}, p_{n-1}, R_{n-1})$ verifying \eqref{eq:hypothesis1}, \eqref{eq:hypothesis2} and \eqref{eq:hypothesis3} with $n-1$ in place of $n$ we aim to construct a new triplet $(u_{n}, p_{n}, R_{n})$ verifying the same set of estimates so that for the velocity increment $u_n-u_{n-1}$ it holds \eqref{eq:main_iteration_w_n}. The main part of the proof is to find the proper velocity increment, which consists of the following several steps. We first set up all the constants except $a$ in the beginning to convince the reader there is no loophole. Then we mollify the solution triplet $(u_{n-1}, p_{n-1}, R_{n-1})$ to obtain a mollified solution $(\ul, \pl, \Rl)$ and we derive several standard estimates for $(\ul, \pl, \Rl)$. The goal is then to find the perturbation $w_n$ so that $u_{n} = \ul +w_n$ verifies the inductive hypothesis \eqref{eq:hypothesis2} and \eqref{eq:hypothesis3}. After mollification, we introduce a partition to properly decompose the Reynolds stress $\Rl$ so that Lemma \ref{lemma:mikado_R} applies. Once the decomposition of the Reynolds stress $\Rl$ is done, the velocity perturbations $w_n^p$ and $w_n^c$ can be defined, where $w_n^p$ is the principle part and $w_n^c$ is a lower order correction to ensure the divergence-free condition of $w_n$. We then derive a number of estimates for $w_n^p$ and $w_n^c$ from their definitions. After all the preparations, the new Reynolds stress $R_n$ can be solved from a divergence equation, and the last inductive hypothesis \eqref{eq:hypothesis1} follows by using all the established estimates for $\ul$, $w_n^p$ and $w_n^c$. 

\stepcounter{steps}
\subsection*{Step \arabic{steps}: Set up constants.}
We fix all the constants $\beta$, $b$, $\alpha$ appeared in the statement except $a$. The value of $a$ will be required to be larger several times in the following, mainly to absorbed various implicit constants from computations.
\begin{itemize} 
\item First, let $\beta=\frac{1}{200}$, $b=5$ and $\alpha= 10^{-6}$ regardless of the dimensions $d\geq 4$. 
\item Second, we define respectively the concentration and oscillation parameters $\mu_n,\sigma_{n}$ as
\begin{align}\label{eq:def_mu_n_sigma_n}
\mu_n = \l_n^{\frac{3}{4}} \quad\text{and} \quad \sigma_n = \l_n^{\frac{1}{4}} .
\end{align} 
\item Third, let $l>0$ be as 
\begin{equation}\label{eq:def_length}
l=   \frac{\delta_{n }^{1/2} }{\delta_{n-1 }^{1/2} \l_{n-1 }^{1+\alpha} }  .
\end{equation}
Then $\l_{n-1} \leq l^{-1} \leq \l_{n-1}^{1+ \frac{1}{40}} $ and there exists $\theta>0$ so that $l^{-1-\theta} \leq \l_n$.
\end{itemize}
 
Let us explain the role of each parameter. The parameter $\beta$ is the resulting regularity of the final solution $u$, i.e. $u\in H^{\beta}$. Since we aim to prove the existence of stationary weak solutions in $L^2$, we need the compactness of some embedding $H^{\beta} \hookrightarrow L^2 $ to obtain a convergent sequence.

The parameter $b$ is the exponential frequency gap between each step of the iteration. Recall that the frequency is given by $\l_n =a^{(b^n)}$. So $\l_{n-1}^b =\l_n $ and hence $\l_{n-1} \ll \l_n$.

Compared with $\beta$, the parameter $\alpha>0$ is very small which will be used for absorbing lower order factors in conjunction with a large $a\gg 1$. More precisely, in the sequel we often use the following fact: for any constant $C>1$ there is a sufficiently large $a>0$ so that $C \l_{n-1}^{-\alpha} \leq 1$ (recall that $n\geq 2$ in the proof so in view of \eqref{eq:def_for_l_n_d_n} this is possible).
\stepcounter{steps}
\subsection*{Step \arabic{steps}: Mollification.}

This step is to fix the problem of possible loss of derivative which is usual in convex integration. Since by mollifying we inevitably introduce new errors from the nature of noncommutativity between \eqref{eq:NSRE} and mollification, the length scale that we use for the mollification is expected to be larger than $\l_{n-1}$ so that the new errors are not too large. This is the reason that we choose the length scale $l$ as in \eqref{eq:def_length}. Fix a standard mollifying kernel $\eta$ in space and define the mollifications of the velocity field, pressure scaler and Reynolds stress
\begin{align*}
&\ul = \eta_l * u_{n-1}\\
&\pl = \eta_l * p_{n-1} +   \eta_l * (u_{n-1}^2) - |\eta_l* u_{n-1}|^2 \\
& \Rl = \eta_l * R_{n-1} + \eta_l * (u_{n-1}  \mathring{\otimes}  u_{n-1} ) - ( \eta_l* u_{n-1}  \mathring{\otimes}  \eta_l* u_{n-1})    
\end{align*}
where $l$ is the length parameter defined by \eqref{eq:def_length} and $\mathring{\otimes}$ is the trace-less tensor product $f\mathring{\otimes} g:= f_i g_j - \delta_{ij} f_i g_j $.

And hence we have a mollified solution triplet $ (u_l,p_l,\overline{R}_l )$ verifying the following mollified system:
\begin{equation}\label{eq:NSRE_mollified}\tag{Mollified-NSR}
\begin{cases}
- \Delta  \ul +\D (\ul \otimes \ul) +\nabla \pl =\D \Rl &\\
 \D \ul =0 &
\end{cases} 
\end{equation}

We can derive the following set of estimates by standard properties of mollifier. 
\begin{lemma}\label{lemma:mollidied_u_estimate}
For any $m \in \NN$ we have
\begin{align}
\|\ul  - u_{n-1} \|_{2}  & \lesssim \delta_{n }^{\frac{1}{2}} \l_{n-1}^{- \alpha}\\ 
\|\nabla^{m+1 } \ul \|_2 & \lesssim l^{-m} \delta_{n-1 }^{\frac{1}{2}} \l_{n-1}\\
\|\nabla^m \Rl \|_1 & \lesssim l^{-m}\delta_{n } \l_{n-1}^{-2 \alpha} ,
\end{align}
where all implicit constants are independent of $n$ and $l$.
\end{lemma}
\begin{proof}
By the hypothesis \eqref{eq:hypothesis3} and the obvious estimate for mollifier, we have that
\begin{align*}
 \|\ul  - u_{n-1} \|_{2} \leq l^{} \| \nabla u_{n-1} \|_2 \lesssim \delta_{n-1}^{\frac{1}{2}} \l_{n-1} l \lesssim \delta_{n }^{\frac{1}{2}} \l_{n-1}^{- \alpha}
\end{align*}
where we have used \eqref{eq:def_length} to get the last inequality.

Again by the standard property of mollifier, we obtain that
\begin{align*}
\|\nabla^{m+1 } \ul \|_2 \lesssim  l^{-m} \|\nabla u_{n-1} \|_2 \lesssim \delta_{n-1}^{\frac{1}{2}} \l_{n-1} l^{-m}.
\end{align*}

The last inequality for $\Rl$ follows from the Constantin-E-Titi commutator estimate, Proposition \ref{prop:CET_commutator} in the Appendix \ref{appendix:CET}:
\begin{align*}
\|\nabla^m \Rl \|_1 \lesssim  l^{-m }\|R_{n-1} \|_1 + \|\nabla u_{n-1} \|_2^2 l^{2-m} .
\end{align*}
Using the definition of the length scale \eqref{eq:def_length} and \eqref{eq:hypothesis3} we have
\begin{align*}
\|\nabla u_{n-1} \|_2^2 l^{2} =   l^{2} \delta_{n -1}    \l_{n-1}^{2}\leq   l^{-m}\delta_{n } \l_{n-1}^{-2 \alpha}  .
\end{align*}
So it follows that
\begin{align*}
\|\nabla^m \Rl \|_1 \lesssim l^{-m }\|R_{n-1} \|_1+  l^{-m}\delta_{n } \l_{n-1}^{-2 \alpha}\lesssim   l^{-m}\delta_{n } \l_{n-1}^{-2 \alpha}.
\end{align*}
\end{proof}

\begin{remark}
It is worth noting that the constants in Lemma \ref{lemma:mollidied_u_estimate} depend on $m$. This type of dependence will appear in the sequel as well. However, we will only require these higher order Sobolev norms up to some fixed order throughout the proof.
\end{remark}

\stepcounter{steps}
\subsection*{Step \arabic{steps}: Decompose the Reynolds stress}

Recall that Lemma \ref{lemma:mikado_R} is for symmetric matrices in a given compact subset $\mathcal{B} \subset \mathcal{S}^{d\times d}_+$. Since $ \Rl  $ is measured in $L^1$, we need to properly decompose $\Rl$ so that we are able to use the concentrated Mikado flows. 

Choose a smooth cutoff function $\chi:\RR^{d \times d} \to [0,1]$ so that
\begin{equation}
\chi(x)=
\begin{cases}
1 & \quad \text{if $ |x|\in [0,\frac{3}{4}]$}\\
0 & \quad \text{if $ |x| \geq 1$}
\end{cases}
\end{equation}
Now let $\chi_i(x) = \chi(4^{-i} x )$ for any $i \geq 0$ and define positive cutoff functions $\phi_i:\RR^{d \times d} \to [0,1] $ as
\begin{equation}
(\phi_i)^\frac{1}{2}(x)=
\begin{cases}
\chi_{i} -\chi_{i-1} & \quad \text{if $ i \geq 1$}\\
\chi_{0} & \quad \text{if $ i = 0$}
\end{cases}
\end{equation}
so that we have by telescoping
\begin{align*}
\phi_0^2(x) + \sum_{i\geq 1} \phi^2_i(x) =1.
\end{align*}

Then define the partition for the Reynolds stress
\begin{align}\label{eq:def_R_cutoff}
\chi_{i,n}(x):= \phi_i \Big( \frac{  \overline{R}_l  }{\delta_{n}\l_{n-1}^{-2 \alpha }} \Big) \quad \text{for any $i \in \NN$}.
\end{align}

We are ready to define the velocity increment. First apply Lemma \ref{lemma:mikado_R} with $N=1 $ and then let the principle part of the perturbation be:
\begin{align}\label{eq:def_w_p}
w_n^p(x)=   \sum_{i \geq 0} 2^{i+1} \delta_{n}^\frac{1}{2} \l_{n-1}^{-  \alpha } \chi_{i,n} W_{[i]}^{\mu_n}\Big(\Id - \frac{ \overline{R}_l }{4^{i+1}\delta_{n}\l_{n-1}^{-2 \alpha }}, \sigma_n x   \Big)
\end{align}
where $[i] = i \mod 2$. Since 
$$
4^{i-1 } \delta_{n}\l_{n-1}^{-2 \alpha } \leq | \overline{R}_l |\leq 4^{i }\delta_{n}\l_{n-1}^{-2 \alpha } \quad \text{for all $x \in \Supp \chi_{i,n}$}
$$
we know that
\begin{align*}
 \bigg|\frac{ \overline{R}_l }{4^{i+1}\delta_{n}\l_{n-1}^{-2 \alpha }} \bigg|\leq \frac{1}{4}
\end{align*}
and thus 
\begin{align*}
  \Id-\frac{ \overline{R}_l }{4^{i+1}\delta_{n}\l_{n-1}^{-2 \alpha }} \in \mathcal{B}\subset \mathcal{S}^{d\times d}_+ \quad \text{for all $x \in \TT^d$}.
\end{align*}
So $w_n^p$ is well-defined in view of Lemma \ref{lemma:mikado_R}. In what follows, with a slight abuse of notations we often use the shorthand
\begin{align}\label{eq:def_short_w_p}
w_n^p= \sum_{i} a_{i,k,n}  \psi_{i,k}^{\mu_n}(\sigma_n \cdot)
\end{align}
where it is understood that in the functions $\psi_{i,k}^{\mu_n}$ the lower index $i=0$ when even and $i=1$ when odd and the short notation $a_{i,k,n}$ is defined by
\begin{align}\label{eq:def_a_ikn}
a_{i,k,n}= 2^{i+1} \delta_{n}^\frac{1}{2} \l_{n-1}^{-  \alpha }   \chi_{i,n}   k \Gamma_k \Big( \Id - \frac{ \overline{R}_l }{4^{i+1}\delta_{n}\l_{n-1}^{-2 \alpha }} \Big)  .
\end{align}

Let us show that $w^p_n$ is non-trivial in the following. Recall that $\mathcal{B} \subset \mathcal{S}^{d\times d}_+ $ is the ball of radius $\frac{1}{2}$ centered at $\Id $. So $0 \not \in \mathcal{B}$, which means that
\begin{align}
\text{for any $R\in\mathcal{B} $ there exists at least a $k \in \mathbb{Z}^d$ so that }\Gamma_k   ( R ) \neq 0.
\end{align}
And then it follows that for any $x\in \TT^d$ there exists at least one $\Gamma_k$ so that
\begin{align}\label{eq:gamma_k_nontrivial}
\Gamma_k \Big( \Id - \frac{ \overline{R}_l }{4^{i+1}\delta_{n}\l_{n-1}^{-2 \alpha }} \Big) \neq 0.
\end{align}
Thus due to the partition \eqref{eq:def_R_cutoff} for this particular $k$ there exists $i \in \NN$ so that
\begin{align}\label{eq:a_k_nontrivial}
|a_{i,k,n}|> 0  \quad \text{for some $x \in \TT^d$}
\end{align}
which makes sure the non-triviality of $w_n^p$.

One notices that $w_n^p$ is not divergence-free. However we can fix this by introducing a corrector:
\begin{align}\label{eq:def_w_c}
w^c_n= - |\nabla|^{-1}  \mathcal{R}_j \bigg[ \sum_{i} \big( \D a_{i,k,n} \big) \psi_{i,k}^{\mu_n}(\sigma_n \cdot) \bigg]
\end{align}
where $\mathcal{R}_j$ is the Riesz transform with symbol $ \frac{k_j}{|k|}$ for $1\leq j \leq d$. Again for better exposition we will use the short notation 
\begin{align}\label{eq:def_short_w_c}
w_n^c= - |\nabla|^{-1}  \mathcal{R}_j \sum_{i} b_{i,k,n}  \psi_{i,k}^{\mu_n}(\sigma_n \cdot)
\end{align}
where 
$$
b_{i,k,n}= \D a_{i,k,n}=2^{i+1} \delta_{n}^\frac{1}{2} \l_{n-1}^{-  \alpha }  \D  \Big[  \chi_{i,n}   k \Gamma_k    \Big].
$$
Considering the fact that to the leading order $w^p_n$ is divergence-free, the corrector is expected to be much smaller and will not be of any trouble. It is easy to check 
\begin{align*}
\D w_n =\D w^p_n+\D w^c_n =0
\end{align*}
and thanks to \eqref{eq:gamma_k_nontrivial}, \eqref{eq:a_k_nontrivial}, and \eqref{eq:def_w_c}, $w_n$ is not identically $0$.

Now that we have successfully define the velocity perturbation $w_n$, the velocity at step $n$ is then given by 
\begin{equation}\label{eq:def_u_n_w_n}
 u_{n}= \ul +w_n
\end{equation}

Note that we perturb $\ul$ not $u_{n-1}$ since the mollified velocity field verifies much nicer estimates than $u_{n-1}$. If one instead uses $u_{n-1}+w_n$ then the typical problem of losing derivative appears.

The new Reynolds stress $R_n$ will be computed later via \eqref{eq:NSRE_mollified}.
\stepcounter{steps}
\subsection*{Step \arabic{steps}: Estimate the coefficients \texorpdfstring{$a_{i,k,n}$}{a} and \texorpdfstring{$b_{i,k,n}$}{b}}
We show that the coefficients $a_{i,k,n}$ and $b_{i,k,n}$ have frequency $\sim l^{-1}$ and are of the correct sizes for proving regularity of $w_n$ in the next step. This requires a result on the H\"older norm of composition of functions, Proposition \ref{prop:holder_composition}.
\begin{proposition}\label{prop:coef_a_b}
There exists an index $i_{\max} \lesssim \ln \l_n^{-1}$ so that 
\begin{align*}
a_{i,k,n}=b_{i,k,n}=0\quad \text{for all $i \geq i_{\max}$}.
\end{align*}
Moreover we have for any $m\in \NN$ that
\begin{align*}
\| \nabla^m a_{i,k,n} \|_{L^\infty} & \lesssim  \delta_{n}^\frac{1}{2} \l_{n-1}^{-  \alpha } l^{ -m -2d-2}\\
\| \nabla^m (a_{i,k,n}^2) \|_{L^\infty}& \lesssim  \delta_{n}^\frac{1}{2} \l_{n-1}^{-  \alpha } l^{ -m -2d-2}\\
\| \nabla^m b_{i,k,n} \|_{L^\infty} & \lesssim \delta_{n}^\frac{1}{2} \l_{n-1}^{-  \alpha } l^{ -m -2d-3} 
\end{align*}
and 
\begin{align*}
\| a_{i,k,n} \|_{L^1} +  l \| \nabla a_{i,k,n} \|_{L^1}   & \lesssim   \delta_{n}^{\frac{1}{2}}\l_{n-1}^{-  \alpha }    \\
\| a_{i,k,n} \|_{L^2}   & \lesssim  \delta_{n}^{\frac{1}{2}}\l_{n-1}^{-  \alpha }    \\
\|  b_{i,k,n} \|_{L^1} & \lesssim l^{-1} \delta_{n}^{\frac{1}{2}}\l_{n-1}^{-  \alpha }     .
\end{align*}
\end{proposition}
\begin{proof}
Let $i_{\max}$ be defined as the smallest integer so that  
\begin{equation}\label{eq:def_i_max}
4^{ i_{\max}+2} \delta_{n } \l_{n-1}^{-2\alpha} \geq \|\Rl \|_\infty.
\end{equation}

Then, from the definition of $a_{i,k,n}$ and $b_{i,k,n}$ we know that 
$$
a_{i,k,n}=b_{i,k,n}=0 \quad \text{for all $i \geq i_{\max}$. }
$$
By the Sobolev embedding $ W^{d+1,1 }(\TT^d) \hookrightarrow L^\infty(\TT^d)$ and the estimate from Lemma \ref{lemma:mollidied_u_estimate} that
\begin{align*}
\|\nabla^i \Rl \|_1 \lesssim l^{-i } \delta_{n} \l_{n-1}^{-  2\alpha }
\end{align*} 
we can conclude that
$$
\|\Rl \|_\infty \lesssim \| \Rl \|_{W^{d+1,1}} \lesssim l^{-(d+1)}   \delta_{n} \l_{n-1}^{-  2\alpha }  
$$
which together with \eqref{eq:def_i_max} implies
$$
i_{\max} \lesssim \ln \l_n^{-1}.
$$

We then estimate the H\"older semi-norms. Let $m\in\NN$. By a crude use of product rule the following pointwise bound holds
\begin{align*}
\big|\nabla^{m} a_{i,k,n}\big| & \lesssim 2^{i+1} \delta_{n}^\frac{1}{2}\l_{n-1}^{-  \alpha } \sum_{0 \leq j \leq m} \Big|\nabla^j {\chi }_{i,n}\Big| \Big| \nabla^{ m-j} \Gamma_k\Big(\Id-\frac{ \overline{R}_l }{4^{i+1}\delta_{n}\l_{n-1}^{-2 \alpha }}    \Big) \Big|  .
\end{align*}
Let $E_{i,n} = \Supp{\chi }_{i,n}$. We will estimate each summand on the support set $E_{i,n} $. Using the H\"older estimate for composition of functions, i.e. Proposition \ref{prop:holder_composition} we have
\begin{align}
\Big\| \nabla^{ m-j} \Gamma_k\Big(\Id-\frac{ \overline{R}_l }{4^{i+1}\delta_{n}\l_{n-1}^{-2 \alpha }}    \Big) \Big\|_{L^\infty(E_{i,n})} 
 & \lesssim \Big\| \nabla^{m-j} \frac{ \overline{R}_l }{4^i\delta_{n }  \l_{n-1}^{-2 \alpha } } \Big\|_{L^\infty(E_{i,n})}     \sum_{i \leq m-j}  \Big\|\frac{ \overline{R}_l }{4^i\delta_{n }  \l_{n-1}^{-2 \alpha } }   \Big\|_{L^\infty(E_{i,n})}^{i-1},
\end{align}
where we have put the H\"older norms of $\Gamma_k$ into the implicit constant. Since there are only finitely many $\Gamma_k$ (depending only on the dimension $d$), this is allowable.
And then we notice that on $E_{i,n} $ it holds that
\begin{align*}
\Rl \sim 4^{i+1}\delta_{n }\l_{n-1 }^{-2\alpha}  \quad \text{for all $x \in E_{i,n}$ }.
\end{align*}
Taking this into account we obtain
\begin{align}\label{eq:a_ikn_gamma}
\Big\| \nabla^{ m-j} \Gamma_k\Big(\Id-\frac{ \overline{R}_l }{4^{i+1}\delta_{n}\l_{n-1}^{-2 \alpha }}    \Big) \Big\|_{L^\infty(E_{i,n})} 
 & \lesssim \Big\| \nabla^{m-j} \frac{ \overline{R}_l }{4^i\delta_{n }  \l_{n-1}^{-2 \alpha } } \Big\|_{L^\infty(E_{i,n})}    .
\end{align}

We can proceed similarly for $\nabla^j {\chi }_{i,n}$ to obtain the bound:
\begin{align}
\Big\|\nabla^j {\chi }_{i,n} \Big\|_{L^\infty(E_{i,n})} 
&\lesssim \Big\| \nabla^j \frac{ \overline{R}_l }{4^i\delta_{n }  \l_{n-1}^{-2 \alpha } }   \Big\|_{L^\infty(E_{i,n})} \sum_{i \leq j} \Big\| \frac{ \overline{R}_l }{4^i\delta_{n }  \l_{n-1}^{-2 \alpha } }   \Big\|_{L^\infty(E_{i,n})}^{i-1}  \nonumber\\
&\lesssim \Big\| \nabla^j \frac{ \overline{R}_l }{4^i\delta_{n }  \l_{n-1}^{-2 \alpha } }   \Big\|_{L^\infty(\TT^d) } \label{eq:a_ikn_chi}.
\end{align}
Moreover thanks to the Sobolev embedding $W^{d+1,1}(\TT^d) \hookrightarrow L^\infty(\TT^d)$ it follows from \eqref{eq:a_ikn_gamma} and Lemma \ref{lemma:mollidied_u_estimate} that
\begin{align}
\Big\| \nabla^{ m-j} \Gamma_k\Big(\Id-\frac{ \overline{R}_l }{4^{i+1}\delta_{n}\l_{n-1}^{-2 \alpha }}    \Big) \Big\|_{L^\infty(E_{i,n})} 
& \lesssim \Big\|   \frac{ \overline{R}_l }{4^{i+1}\delta_{n }  \l_{n-1}^{-2 \alpha } }      \Big\|_{W^{m-j+d+1,1}(\TT^d) }  \nonumber\\
&  \lesssim 4^{-i-1}l^{-m+j-d-1}
\end{align}
and from \eqref{eq:a_ikn_chi} Lemma \ref{lemma:mollidied_u_estimate} that
\begin{align}
\Big\|\nabla^j {\chi }_{i,n} \Big\|_{L^\infty(E_{i,n})} 
&\lesssim  \Big\|   \frac{ \overline{R}_l }{4^i\delta_{n }  \l_{n-1}^{-2 \alpha } }   \Big\|_{W^{j+d+1,1}(\TT^d)} \nonumber \\
&  \lesssim 4^{-i-1}l^{-j-d-1}.
\end{align}

Inserting these bounds into the estimate for $\nabla^{m} a_{i,k,n}$ we have
\begin{align*}
\| \nabla^{m} a_{i,k,n} \|_{L^\infty(\TT^d)} & \lesssim 2^{i+1} \delta_{n}^\frac{1}{2}\l_{n-1}^{-  \alpha } \sum_{0 \leq j \leq m} \Big\|\nabla^j {\chi }_{i,n}\Big\|_{L^\infty(E_{i,n})} \Big\| \nabla^{ m-j} \Gamma_k\Big(\Id-\frac{ \overline{R}_l }{4^{i+1}\delta_{n}\l_{n-1}^{-2 \alpha }}    \Big) \Big\|_{L^\infty(E_{i,n})} \\
&\lesssim       \delta_{n}^\frac{1}{2} \l_{n-1}^{-  \alpha } l^{ -m -2d-2}.
\end{align*}

Observing that
\begin{align*}
\nabla^{m} (a_{i,k,n})^2  & \lesssim 4^{i+1} \delta_{n} \l_{n-1}^{-  2 \alpha } \sum_{0 \leq j \leq m} \Big|\nabla^j {\chi }^2_{i,n}\Big| \Big| \nabla^{ m-j} \Gamma_k^2 \Big(\Id-\frac{ \overline{R}_l }{4^{i+1}\delta_{n}\l_{n-1}^{-2 \alpha }}    \Big) \Big|,  
\end{align*}
and that
\begin{align*}
\nabla^{m}  b_{i,k,n}   & \lesssim 2^{i+1} \delta_{n}^\frac{1}{2} \l_{n-1}^{-    \alpha } \sum_{0 \leq j \leq m+1} \Big|\nabla^{j} {\chi }^2_{i,n}\Big| \Big| \nabla^{ m-j} \Gamma_k^2 \Big(\Id-\frac{ \overline{R}_l }{4^{i+1}\delta_{n}\l_{n-1}^{-2 \alpha }}    \Big) \Big|  ,
\end{align*}
where the H\"older norms of all factors in the summation have been estimated,
we can conclude without proof that
\begin{align*}
\| \nabla^{m} (a_{i,k,n})^2 \|_{L^\infty(\TT^d)} 
&\lesssim       \delta_{n}^\frac{1}{2} \l_{n-1}^{-  \alpha } l^{ -m -2d-2},
\end{align*}
and 
\begin{align*}
\| \nabla^{m} b_{i,k,n} \|_{L^\infty(\TT^d)} 
&\lesssim       \delta_{n}^\frac{1}{2} \l_{n-1}^{-  \alpha } l^{ -m -2d-3}.
\end{align*}
It remains to show $L^1$ and $L^2$ bounds. Since these bounds are more delicate than the ones in $L^\infty$-based norms, we estimate them in a more precise manner. By the definition of $\chi_{i,n}$ we have
\begin{align*}
\|a_{i,k,n} \|_1 &\leq  2^{i+1} \delta_{n}^\frac{1}{2}\l_{n-1}^{-  \alpha } \int_{\TT^d} \chi_{i,k,n}dx \\
& \lesssim 2^{i+1} \delta_{n}^\frac{1}{2}\l_{n-1}^{-  \alpha } |\Supp \chi_{i,n}|.
\end{align*}
and
\begin{align*}
\|a_{i,k,n} \|_2^2 &\leq  4^{i+1} \delta_{n} \l_{n-1}^{-  2\alpha } \int_{\TT^d} \chi_{i,k,n}^2 dx \\
& \lesssim 4^{i+1} \delta_{n} \l_{n-1}^{-  2\alpha } |\Supp \chi_{i,n}|.
\end{align*}
By the decomposition of the Reynolds stress $\Rl$ we know that
\begin{align*}
|\Supp \chi_{i,n}| \leq \Big|  \{x\in \TT^d: \Big| \frac{ \overline{R}_l }{\delta_{n}\l_{n-1}^{-2 \alpha }} \Big|\geq 4^{i-1}   \} \Big|
\end{align*}
from which it follows by the Chebyshev inequality that
\begin{align*}
|\Supp \chi_{i,n}| \leq  4^{-i+1}.
\end{align*}
So
\begin{align*}
\|a_{i,k,n} \|_1 & \lesssim 2^{-i} \delta_{n}^\frac{1}{2}\l_{n-1}^{-  \alpha }.
\end{align*}
and respectively
\begin{align*}
\|a_{i,k,n} \|_2^2 & \lesssim \delta_{n} \l_{n-1}^{-  2\alpha }.
\end{align*}
To bound $\nabla a_{i,k,n}$ we obtain first by the product rule and chain rule:
\begin{align*}
 |\nabla a_{i,k,n} | & \lesssim 2^{i+1} \delta_{n}^\frac{1}{2}\l_{n-1}^{-  \alpha }  \bigg(    | {\chi }_{i,n} | | \nabla  \Gamma_k |  \Big|\frac{ \nabla \overline{R}_l }{4^{i+1}\delta_{n}\l_{n-1}^{-2 \alpha }} \Big|  +      |   \Gamma_k | |\nabla  \phi_i | \Big|\frac{ \nabla \overline{R}_l }{4^{i+1}\delta_{n}\l_{n-1}^{-2 \alpha }} \Big|         \bigg) .
\end{align*}
Then using the obvious bounds 
\begin{align*}
& | {\chi }_{i,n} | \leq 1, \quad |\nabla  \phi_i | \lesssim 4^{-i}\\
& |   \Gamma_k | \lesssim 1, \quad |  \nabla \Gamma_k | \lesssim 1 
\end{align*}
we can estimate $L^1$ norm as follows:
\begin{align*}
\|\nabla a_{i,k,n} \|_1 &\lesssim  2^{i+1} \delta_{n}^\frac{1}{2}\l_{n-1}^{-  \alpha } \int_{\TT^d}   \Big|  \frac{\nabla \overline{R}_l }{4^{i+1}\delta_{n}\l_{n-1}^{-2 \alpha }}     \Big| dx  \\
&\lesssim 2^{-i} \delta_{n}^{- \frac{1}{2}}\l_{n-1}^{  \alpha } \|\nabla \Rl \|_1\\
&\lesssim l^{-1} \delta_{n}^{\frac{1}{2}}\l_{n-1}^{-  \alpha }    ,
\end{align*}
where we have used $\|\nabla \Rl \|_1 \lesssim l^{-1 } \delta_{n} \l_{n-1}^{-  2\alpha }$ from Lemma \ref{lemma:mollidied_u_estimate}.

From the definition of $b_{i,k,n}$ it is clear that $b_{i,k,n}$ also verifies bound:
\begin{align*}
\| b_{i,k,n} \|_1 \lesssim l^{-1} \delta_{n}^{\frac{1}{2}}\l_{n-1}^{-  \alpha }   . 
\end{align*}

\end{proof}

\stepcounter{steps}
\subsection*{Step \arabic{steps}: Estimate the velocity perturbation}

We summarize the regularity properties of $w_n^p$ in the below proposition. 
\begin{proposition}[Regularity of $w_n^p$]\label{prop:w_p}
There exists $a_0>0$ sufficiently large so that for any $a \geq a_0$ the principle part of velocity increment defined by \eqref{eq:def_w_p} verifies
\begin{align}
\|w_n^p \|_2 + \frac{1}{\l_{n}} \|\nabla w_n^p \|_2  &\leq \frac{1}{8} \delta_{n}^{\frac{1}{2}}  \label{eq:w_p_L2}\\
\|w_n^p \|_1 + \frac{1}{\l_n} \|\nabla w_n^p \|_1 & \leq \frac{1}{8} \delta_{n}^\frac{1}{2}    \mu_{n}^{\frac{-d+1}{2}}  \label{eq:w_p_L1}.
\end{align}
\end{proposition}

\begin{proof}
Thanks to Proposition \ref{prop:coef_a_b}, we know that $w_n^p$ consists of finitely many concentrated Mikado flows:
\begin{align*}
w_n^p= \sum_{0\leq i\leq i_{\max}} \sum_{k} a_{i,k,n}  \psi_{i,k}^{\mu_n}(\sigma_n \cdot).
\end{align*}

To show the bound for $\|w_n^p \|_2$ it suffices to show that 
\begin{align}
\|a_{i,k,n}  \psi_{i,k}^{\mu_n}(\sigma_n \cdot) \|_2  & \lesssim \delta_{n}^\frac{1}{2} \l_{n-1}^{ - \alpha}   \label{eq:w_p_a_ik_L2}
\end{align}
The reason is that the extra factor $\l_{n-1}^{ - \alpha}$ can be used to absorbed the logarithmic error causing by $i_{\max}$ and any constant factors provided that $a$ is sufficiently large.

Since $\psi_{i,k}^{\mu_n}(\sigma_n \cdot) $ is $\sigma_n^{-1}\TT^d$-periodic and by Proposition \ref{prop:coef_a_b}
$$
\| \nabla^m a_{i,k,n} \|_{L^\infty}   \lesssim  \delta_{n}^\frac{1}{2} \l_{n-1}^{-  \alpha } l^{ -m -2d-2}
$$ 
we can apply the first part of Proposition \ref{prop:inversediv_commutator2} with $C_a=  \delta_{n}^\frac{1}{2} \l_{n-1}^{-  \alpha } l^{   -2d-2}$, $\mu =l^{-1}$, $\sigma=\sigma_{n}$ to obtain that
\begin{align*}
\|a_{i,k,n}  \psi_{i,k}^{\mu_n}(\sigma_n \cdot) \|_2  & \lesssim \|a_{i,k,n}\|_2 \| \psi_{i,k}^{\mu_n}(\sigma_n \cdot) \|_2 +C_a\sigma_{n}^{-100d}  .
\end{align*}
The second term appeared on the right is essentially a small error term. Indeed since $C_a \sigma_{n}^{-100d} \ll \l_n^{-10d}$ due to the fact that $l^{-1} \leq \sigma_{n}^2$ we obtain
\begin{align*}
\|a_{i,k,n}  \psi_{i,k}^{\mu_n}(\sigma_n \cdot) \|_2  &\lesssim \delta_{n}^\frac{1}{2} \l_{n-1}^{ - \alpha}  .
\end{align*}
And hence by taking $a$ sufficiently large, \eqref{eq:w_p_L2} can be obtained:
\begin{align*}
\| w_n^p\|_2  &\leq   \frac{1}{16}  \delta_n^\frac{1}{2}  
\end{align*}
To show the bound for $\|w_n^p \|_1$ we simply observe that 
\begin{align*}
\big| \Supp a_{i,k,n}  \psi_{i,k}^{\mu_n}(\sigma_n \cdot)  \big| \leq \big| \Supp    \psi_{i,k}^{\mu_n}(\sigma_n \cdot)  \big|\lesssim \mu_n^{-(d-1)}
\end{align*}
and thus by Jensen's inequality
\begin{align*}
\|a_{i,k,n}  \psi_{i,k}^{\mu_n}(\sigma_n \cdot) \|_1 \lesssim \|a_{i,k,n}  \psi_{i,k}^{\mu_n}(\sigma_n \cdot) \|_2 \mu_n^{-\frac{d-1}{2} }
\end{align*}
Again by taking $a$ sufficiently large, \eqref{eq:w_p_L1} can be obtained:
\begin{align*}
\| w_n^p\|_1  &\leq  \frac{1}{16}  \delta_n^\frac{1}{2} \mu^{\frac{-d+1}{2} }.
\end{align*}

Now we turn to estimate $\|\nabla w_n^p \|_p$ for $p=1$ or $2$. By the same argument of using small support set and Jensen's inequality, it suffices to only show the bound for $\|\nabla w_n^p \|_2$. Taking derivative on $w_n^p$ we have
\begin{align*}
\nabla w_n^p = \sum_{i,k} \nabla a_{i,k,n} \psi_{i,k}^{\mu_n}(\sigma_n \cdot )   + \sigma_n \sum_{i,k} a_{i,k,n}   \nabla \psi_{i,k}^{\mu_n}(\sigma_n \cdot ) .
\end{align*}
Following the same argument we apply Proposition \ref{prop:inversediv_commutator2} to the two above summands with $C_a=  \delta_{n}^\frac{1}{2} \l_{n-1}^{-  \alpha } l^{   -2d-2}$, $\mu =l^{-1}$, $\sigma=\sigma_{n}$ and then obtain that
\begin{equation*}
\|\nabla a_{i,k,n}  \psi_{i,k}^{\mu_n}(\sigma_n \cdot) \|_2    \lesssim \delta_{n}^\frac{1}{2} l^{-1} \l_{n-1}^{ -  \alpha}   
\end{equation*}
and 
\begin{align*}
\| a_{i,k,n} \nabla  \psi_{i,k}^{\mu_n}(\sigma_n \cdot) \|_2  & \lesssim \delta_{n}^\frac{1}{2} \mu_n \l_{n-1}^{ -  \alpha}  .
\end{align*}
Hence by the relationship $\l_n =\sigma_n \mu_n $ it is obtained that
\begin{align*}
\frac{1}{\l_n}\|\nabla w_n^p\|_2  & \lesssim   \delta_n^\frac{1}{2}  \l_{n-1}^{-  \alpha } 
\end{align*}
By choosing $a$ sufficiently large it holds that
\begin{align*}
\frac{1}{\l_n}\|\nabla w_n^p\|_2  & \leq   \frac{1}{16}\delta_n^\frac{1}{2}    .
\end{align*}
\end{proof}
Next, we turn to estimate the correction part of the velocity $w_n^c$. As expected $w_n^c$ is much smaller than $w^p_n$.

\begin{proposition}[Regularity of $w_n^c$]\label{prop:w_c}
There exists $a_0>0$ sufficiently large so that for any $a \geq a_0$ the correction part of velocity increment defined by \eqref{eq:def_w_p} verifies
\begin{align}
\|w^c_n \|_2 +\frac{1}{\l_n} \|\nabla w^c_n \|_2 & \leq \frac{1}{8}\delta_{n}^\frac{1}{2}    l^{-1}  \sigma_{n}^{-1}      \label{eq:w_c_L2}\\
\|w^c_n \|_1 +\frac{1}{\l_n} \|\nabla w^c_n \|_1 & \leq \frac{1}{8} \delta_{n}^\frac{1}{2}     \mu_{n}^{\frac{-d+1}{2}}  l^{-1}  \sigma_{n}^{-1}     \label{eq:w_c_L1}.
\end{align}
\end{proposition}
\begin{proof}
Observe that the definition of $w^c_n$ involves Riesz transform which is only bounded $L^p \to L^p$ when $1<p<\infty$. So to resolve this issue let us fix a parameter $s>1$ sufficiently close to $1$ such that
\begin{equation}\label{eq:s_close_to_1}
\mu_n^{\frac{d-1}{1}  -\frac{d-1}{s}} \leq \l_{n-1}^{\frac{1}{2}\alpha}.
\end{equation}
And we instead estimate the $L^s$ norm rather than $L^1$ norm.
Let us first prove the bounds for $\nabla w^c_n $ as in this case we can follow along the lines of Proposition \ref{prop:w_p}. By the $L^p \to L^p$ boundedness of the Riesz transform for any $1 < p <\infty$ we notice that 
\begin{align}\label{eq:w_c_b_psi_Lp}
\| \nabla w^c_n  \|_{L^p(\TT^d)} \lesssim  \bigg\| \sum_{i \leq i_{\max}} b_{i,k,n}  \psi_{i,k}^{\mu_n}(\sigma_n \cdot) \bigg\|_{L^p(\TT^d) } .
\end{align}
Since by Proposition \ref{prop:coef_a_b} 
$$
\| \nabla^m b_{i,k,n} \|_{L^\infty}   \lesssim \delta_{n}^\frac{1}{2} \l_{n-1}^{-  \alpha } l^{ -m -2d-3} 
$$
so applying the first part of Proposition \ref{prop:inversediv_commutator2} with $C_a=  \delta_{n}^\frac{1}{2} \l_{n-1}^{-  \alpha } l^{   -2d-3}$, $\mu =l^{-1}$, $\sigma=\sigma_{n}$ and after simplifying one obtains 
\begin{align}\label{eq:w_c_b_psi_L2}
\big\|   b_{i,k,n}  \psi_{i,k}^{\mu_n}(\sigma_n \cdot) \big\|_{2}  & \lesssim \delta_{n}^\frac{1}{2} \l_{n-1}^{-  \alpha }     l^{-1}  
\end{align}
and then by small support of $\psi_{i,k}^{\mu_n}(\sigma_n x) $, namely
$$
\Big| \Supp b_{i,k,n}  \psi_{i,k}^{\mu_n}(\sigma_n x)  \Big| \leq \Big| \Supp  \psi_{i,k}^{\mu_n}(\sigma_n x)  \Big| \lesssim \mu_n^{-(d-1)}
$$
and Jensen's inequality we have
\begin{align}\label{eq:w_c_b_psi_Ls}
\big\|   b_{i,k,n}  \psi_{i,k}^{\mu_n}(\sigma_n \cdot) \big\|_{s }  & \lesssim  \mu_n^{\frac{d-1}{1}  -\frac{d-1}{s}} \big\|   b_{i,k,n}  \psi_{i,k}^{\mu_n}(\sigma_n \cdot) \big\|_{2 }  \nonumber\\
& \lesssim\delta_{n}^\frac{1}{2} \l_{n-1}^{-  \alpha }     l^{-1} \mu_n^{\frac{-d+1}{2}} \l_{n-1}^{-\frac{1}{2}\alpha}
\end{align}
where we have used the fact that $\mu_n^{\frac{d-1}{1} -\frac{d-1}{s} }\leq \l_{n-1}^{\frac{1}{2}\alpha} $.

Since $\sigma_{n} =\l_n^{\frac{1}{4}} \leq \l_n$ from \eqref{eq:w_c_b_psi_L2} and \eqref{eq:w_c_b_psi_Lp} we get
\begin{align*}
\frac{1}{\l_n}\|\nabla w^c_n \|_2 & \leq \frac{1}{16}\delta_{n}^\frac{1}{2}    l^{-1}  \sigma_{n}^{-1}       
\end{align*}
as long as $a$ is sufficiently large. 

To recover the $L^1$ bound for $\nabla w^c_n $ we simply first bound $L^1$ norm by its $L^s$ norm:
$$
\|\nabla w^c_n \|_1\leq \|\nabla w^c_n \|_s 
$$
and it follows that
$$
\|\nabla w^c_n \|_1 \lesssim \sum_{i} \big \|  b_{i,k,n}  \psi_{i,k}^{\mu_n}(\sigma_n \cdot) \big \|_{s} .
$$
Then taking $a$ sufficient large and using \eqref{eq:w_c_b_psi_Ls} we can ensure that
\begin{align*}
\|\nabla w^c_n \|_1   \leq  \frac{1}{16}\delta_{n}^\frac{1}{2} \mu_n^{\frac{-d+1}{2}}     l^{-1} 
\end{align*}
which implies that
\begin{align*}
\frac{1}{\l_n}\|\nabla w^c_n \|_1   \leq  \frac{1}{16}\delta_{n}^\frac{1}{2} \mu_n^{\frac{-d+1}{2}}   l^{-1}  \l_{n}^{-1}   .
\end{align*}

It remains to prove the estimates of $\|  w^c_n  \|_p $ for $p=1,2$. It follows from the $L^p$ boundedness of Riesz transform that
\begin{align*}
 \| w^c_n\|_{2 } & \lesssim \sum_{i,k} \big\| |\nabla|^{-1}   \big[b_{i,k,n}  \psi_{i,k}^{\mu_n}(\sigma_n \cdot)  \big] \big\|_{2 } \\
\| w^c_n \|_{s}  &\lesssim \sum_{i,k} \big\|  |\nabla|^{-1}    \big[b_{i,k,n}  \psi_{i,k}^{\mu_n}(\sigma_n \cdot) \big] \big\|_{s}.
\end{align*}
Therefore it suffices to derive suitable estimates for the functions 
\begin{align*}
 |\nabla|^{-1}    \big[b_{i,k,n}  \psi_{i,k}^{\mu_n}(\sigma_n \cdot)  \big] 
\end{align*}
where we note that $b_{i,k,n}  \psi_{i,k}^{\mu_n}(\sigma_n \cdot) $ has zero-mean since
\begin{align*}
\int_{\TT^d }  b_{i,k,n}  \psi_{i,k}^{\mu_n}(\sigma_n x)  ~ dx = \int_{\TT^d }  \D a_{i,k,n}  \psi_{i,k}^{\mu_n}(\sigma_n x)   ~ dx =0.
\end{align*}
So, thanks to Proposition \ref{prop:coef_a_b} the assumptions in Proposition \ref{prop:inversediv_commutator2} are fulfilled:
\begin{align*}
\| \nabla^m b_{i,k,n} \|_{L^\infty} & \lesssim \delta_{n}^\frac{1}{2} \l_{n-1}^{-  \alpha } l^{ -m -2d-3} 
\end{align*}
and then we can obtain from the second part of Proposition \ref{prop:inversediv_commutator2}
\begin{align}\label{eq:inverse_nabla_Pb_psi_L2}
\big\| |\nabla|^{-1}   \big[b_{i,k,n}  \psi_{i,k}^{\mu_n}(\sigma_n \cdot)  \big] \big\|_{2 } 
&\lesssim \sigma_{n}^{-1} \big\|        b_{i,k,n}  \psi_{i,k}^{\mu_n}(\sigma_n \cdot)    \big\|_{2 }  + \sigma_{n}^{-10d}
\end{align}
and
\begin{align}\label{eq:inverse_nabla_Pb_psi_Ls}
\big\|  |\nabla|^{-1}    \big[b_{i,k,n}  \psi_{i,k}^{\mu_n}(\sigma_n \cdot)  \big] \big\|_{s} 
& \lesssim \sigma_{n}^{-1} \big\|       b_{i,k,n}  \psi_{i,k}^{\mu_n}(\sigma_n \cdot)    \big\|_{s }  + \sigma_{n}^{-10d}.
\end{align}
where we have used the fact that here $\delta_{n}^\frac{1}{2} \l_{n-1}^{-  \alpha } l^{   -2d-3} \sigma_{n}^{-100d} \ll \sigma_{n}^{-10d}$.
Then it follows again from the first part of Proposition \ref{prop:inversediv_commutator2} that
\begin{align}\label{eq:Pb_psi_L2L2}
\big\|   b_{i,k,n}  \psi_{i,k}^{\mu_n}(\sigma_n \cdot)    \big\|_{2 } \lesssim  \delta_{n}^\frac{1}{2} \l_{n-1}^{-  \alpha }    \sigma_{n}^{-1} l^{-1} +\sigma_{n}^{-10d},
\end{align}
which by Jensen's inequality and the small support of $ \psi_{i,k}^{\mu_n}(\sigma_n x) $ also implies that
\begin{align}\label{eq:Pb_psi_LsbyL2}
\big\|   b_{i,k,n}  \psi_{i,k}^{\mu_n}(\sigma_n \cdot)    \big\|_{s } \lesssim  \mu_n^{ \frac{d-1}{s} -\frac{d-1}{2}} \delta_{n}^\frac{1}{2} \l_{n-1}^{-  \alpha }    \sigma_{n}^{-1} l^{-1} +\sigma_{n}^{-10d},
\end{align}
So putting together \eqref{eq:s_close_to_1}, \eqref{eq:inverse_nabla_Pb_psi_L2}, \eqref{eq:inverse_nabla_Pb_psi_Ls}, and \eqref{eq:Pb_psi_L2L2} we have
\begin{align*}
 \| w^c_n\|_{2 } & \lesssim \sum_{i,k} \delta_{n}^\frac{1}{2} \l_{n-1}^{-  \alpha }    \sigma_{n}^{-1} l^{-1} \\
\| w^c_n \|_{s}  &\lesssim \sum_{i,k} \delta_{n}^\frac{1}{2} \l_{n-1}^{- \frac{1}{2} \alpha }   \mu_{n}^{\frac{-d+1}{2}} \sigma_{n}^{-1}  l^{-1}  .
\end{align*}
Finally using the extra factors $\l_{n-1}^{-  \alpha }$  and $\l_{n-1}^{- \frac{1}{2} \alpha }$ to absorb any logarithmic and constant factors we can get rid of the summation in $i,k$ to obtain that
\begin{align*}
 \| w^c_n\|_{2 } & \leq \frac{1}{16} \delta_{n}^\frac{1}{2}       l^{-1}  \sigma_{n}^{-1}   
\end{align*}
and that
\begin{align*}
 \| w^c_n \|_{1}  \leq \| w^c_n \|_{s} \leq \frac{1}{16} \delta_{n}^\frac{1}{2}     \mu_{n}^{\frac{-d+1}{2}} l^{-1}  \sigma_{n}^{-1}  .  
\end{align*}
\end{proof}

\begin{remark}
It seems that one should be able to gain a factor of $l^{-1} \l_n^{-1}$ rather than $l^{-1} \sigma_n^{-1}$ since $\psi_{i,k}^{\mu_n}(\sigma_n \cdot)$ has frequency $\l_n$. Such improvement can be obtained by carefully choosing the profile function $\psi$ in Section \ref{Section:mikado} with vanishing moments up to a sufficiently high order, which may be useful in the future study of constructing solutions with better regularity.
\end{remark}

\stepcounter{steps}
\subsection*{Step \arabic{steps}: Check \texorpdfstring{\eqref{eq:main_iteration_w_n}}{2.2} and hypothesis \texorpdfstring{\eqref{eq:hypothesis2}}{H2} and \texorpdfstring{\eqref{eq:hypothesis3}}{H3}}

Let us first check \eqref{eq:main_iteration_w_n}. 
Since by the definition of $u_n$, namely \eqref{eq:def_u_n_w_n} we have
\begin{align*}
\|u_{n } -u_{n-1} \|_2 &\leq \|w_{n } \|_2+ \|\ul -u_{n-1} \|_2
\end{align*}
it suffices to estimate
\begin{align*}
\|u_{n } -u_{n-1} \|_2 &\leq \|w_{n }^p \|_2+\|w_{n }^c \|_2+ \|\ul -u_{n-1} \|_2.
\end{align*}
From the estimates \eqref{eq:w_p_L2} and \eqref{eq:w_c_L2} and Lemma \ref{lemma:mollidied_u_estimate} we know that
\begin{align*}
\|w_{n }^p \|_2+\|w_{n }^c \|_2 &\leq   \frac{1}{4} \delta_{n}^{\frac{1}{2}  }  \\
\|\ul -u_{n-1} \|_2 & \leq  C \delta_{n}^\frac{1}{2} \l_{n-1}^{-  \alpha } .
\end{align*}
Taking a  sufficiently large $a$, we can arrange that
\begin{align*}
\|u_{n } -u_{n-1} \|_2 &\leq  \frac{1}{4} \delta_{n}^{\frac{1}{2}  } .
\end{align*}
and therefore
\begin{align}\label{eq:u_n_u_n-1}
\|u_{n } -u_{n-1} \|_2 &\leq \frac{1}{2} \delta_{n}^{\frac{1}{2}  }.
\end{align}

We will bound the term $\nabla(u_{n } -u_{n-1}) $ in almost the same way. As before we first obtain from the definitions of $u_n$ and $w_n$ that
\begin{align*}
\|\nabla u_{n } -\nabla u_{n-1} \|_2 & \leq \|\nabla w_{n } \|_2+ \|\nabla \ul\|_2+\| \nabla u_{n-1} \|_2\\
&\leq \| \nabla w_{n }^p \|_2+\|\nabla w_{n }^c \|_2+ \|\nabla \ul\|_2 + \| \nabla u_{n-1} \|_2
\end{align*}

Thanks to estimates \eqref{eq:w_p_L2}, \eqref{eq:w_c_L2} and Lemma \ref{lemma:mollidied_u_estimate}, we obtain
\begin{align*}
\|\nabla u_{n } -\nabla u_{n-1} \|_2 
&\leq \| \nabla w_{n }^p \|_2+\|\nabla w_{n }^c \|_2+ \|\nabla \ul\|_2 + \| \nabla u_{n-1} \|_2\\
& \leq  \frac{1}{4} \l_n \delta_{n}^\frac{1}{2}  +C \delta_{n-1 }^{\frac{1}{2}} \l_{n-1}    
\end{align*}
where $C$ is a constant depending only on the mollifier $\eta$.
Again by choose $a$ sufficiently large we can guarantee that
\begin{align*}
\|\nabla u_{n } -\nabla u_{n-1} \|_2 \leq  \frac{1}{2} \l_n \delta_{n}^\frac{1}{2},
\end{align*}
which together with \eqref{eq:u_n_u_n-1} means \eqref{eq:main_iteration_w_n} is satisfied.

Now we show \eqref{eq:hypothesis2}. First, we obtain the obvious bound
\begin{align*}
\|u_{n} \|_2 & = \|\ul +w_n \|_2 \leq \|u_{n-1}\|_2 + \|\ul -u_{n-1}\|_2 + \| w_n \|_2 .
\end{align*}
Then, from Proposition \ref{prop:w_p} and \ref{prop:w_c}, and Lemma \ref{lemma:mollidied_u_estimate} we see that
\begin{align*}
\|u_{n} \|_2 & \leq \|u_{n-1}\|_2 + \|\ul -u_{n-1}\|_2 + \| w_n \|_2 \\
& \leq 1-\delta_{n-1}^\frac{1}{2} + C\delta_{n }^{\frac{1}{2}} \l_{n-1}^{- \alpha}+ \frac{1}{2}  \delta_{n}^\frac{1}{2}  .
\end{align*}
where again $C$ is some constant depending only on the mollifier $\eta$.
Now choosing $a$ sufficiently large depending on $b$ so that
\begin{align*}
&3 \delta_{n}^\frac{1}{2} \leq \delta_{n-1}^\frac{1}{2}  \\
&C  \l_{n-1}^{- \alpha} \leq 1
\end{align*}
we are able to find
\begin{align*}
1-\delta_{n-1}^\frac{1}{2} + C\delta_{n }^{\frac{1}{2}} \l_{n-1}^{- \alpha}+ \frac{1}{2}  \delta_{n}^\frac{1}{2}  & \leq 1-2 \delta_{n}^\frac{1}{2} +  \delta_{n }^{\frac{1}{2}}  + \frac{1}{2}  \delta_{n}^\frac{1}{2}\\
&\leq 1-\delta_{n}^\frac{1}{2}
\end{align*}
and hence we obtain the desire bound \eqref{eq:hypothesis2}:
\begin{align*}
\|u_{n} \|_2 & \leq 1- \delta_{n}^\frac{1}{2}  .
\end{align*}

As for \eqref{eq:hypothesis3}, the proof is very similar. We first obtain
\begin{align*}
\|\nabla u_{n} \|_2 & \leq  \|\nabla \ul \|_2 + \| \nabla w_n \|_2
\end{align*}
and then using Lemma \ref{lemma:mollidied_u_estimate}, estimates \eqref{eq:w_p_L2} and \eqref{eq:w_c_L2} we find that
\begin{align*}
\|\nabla u_{n} \|_2 &\leq \|\nabla \ul \|_2 +\| \nabla w_n^p \|_2+ \| \nabla w_n^c \|_2\\
& \leq C\delta_{n-1 }^{\frac{1}{2}} \l_{n-1}+ \frac{1}{2} \delta_{n}^\frac{1}{2}          \l_n 
\end{align*}
where the constant $C$ depends only on the mollifier $\eta$.
Letting $a$ sufficiently large it can be arranged that
\begin{align*}
C  \l_{n-1}^{1-\beta}\leq  \frac{1}{2}          \l_n^{1-\beta} .
\end{align*}
And then we have
\begin{align*}
C\delta_{n-1 }^{\frac{1}{2}} \l_{n-1}+ \frac{1}{2} \delta_{n}^\frac{1}{2}          \l_n  & \leq  \delta_{n}^{\frac{1}{2}} \l_{n},
\end{align*}
which implies
\begin{align*}
\|\nabla u_{n} \|_2 & \leq \delta_{n}^\frac{1}{2}          \l_n .
\end{align*}
So \eqref{eq:hypothesis3} is also fulfilled.

\stepcounter{steps}
\subsection*{Step \arabic{steps}: Estimate the new Reynolds stress}

Thanks to \eqref{eq:NSRE_mollified}, the new Reynolds stress is defined by the divergence equation:
\begin{equation*}
\begin{aligned}\label{eq:div_form_newR}
\D R_{n} +\nabla P_n  = &  \D \overline{R}_l +  \Delta  w_n +    \D w_n^p \otimes  w_n^p \\
&+ \D ( w_n  \otimes \ul +  \ul \otimes w_n    )\\
&  \quad          + \D ( w_n^c \otimes  w_n^p +  w_n^p \otimes  w_n^c +   w_n^c \otimes  w_n^c) .
\end{aligned}
\end{equation*}

To estimate the $L^1$ norm of $R_{n} $, one needs to somehow invert the divergence. For this purpose we follow the construction given in \cite{1701.08678}. The operator $\mathcal{R}: C^\infty(\TT^d, \RR^d) \to \RR^{d\times d}$ is defined as 
\begin{equation}\label{eq:def_operatorR}
\begin{aligned}
&(  \mathcal{R} f)_{ij} =  \mathcal{R}_{ijk}f_k\\
& \mathcal{R}_{ijk} =  \frac{2-d}{  d-1} \Delta^{-2} \partial_i \partial_j \partial_k+ \frac{-1}{d-1} \Delta^{-1} \partial_k \delta_{ij}  +\Delta^{-1} \partial_i \delta_{jk}  + \Delta^{-1} \partial_j \delta_{ik}.
\end{aligned}
\end{equation}
It is clear that for any $f\in C^\infty(\TT^d)$ the matrix $(  \mathcal{R} f)_{ij}$ is symmetric. Taking the trace we have
\begin{align*}
\Tr    \mathcal{R} f&=  \frac{2-d}{  d -1} \Delta^{-1}    \partial_k f_k + \frac{-d}{d-1} \Delta^{-1} \partial_k   f_k  +   \Delta^{-1}   \partial_k f_k + \Delta^{-1}   \partial_k f_k\\
&=(\frac{2-d}{  d  -1} + \frac{-d}{d-1}+  2) \Delta^{-1}  \partial_k f_k=0
\end{align*}
which means that $\mathcal{R} f$ is also trace-less.

And lastly, we have 
$$
\D    \mathcal{R} f= \partial_j (  \mathcal{R} f)_{ij} =\partial_j \mathcal{R}_{ijk}f_k
$$ 
so by direct computation we can check that
\begin{align*}
\D    \mathcal{R} f=\frac{2-d}{  d -1} \Delta^{-1} \partial_i   \partial_k f_k + \frac{-1}{d-1} \Delta^{-1} \partial_k \partial_i f_k  + \Delta^{-1} \partial_i \partial_k f_k  +   f_i =f_i=f.
\end{align*}

\begin{lemma}\label{lemma:inverse_div}
The operator $\mathcal{R}$ defined by \eqref{eq:def_operatorR} has the following properties. For any $f\in C^\infty_0(\TT^d)$ the matrix $\mathcal{R} f$ is symmetric trace-free and we have
\begin{equation}\label{eq:inverse_div}
\D \mathcal{R} f =f.
\end{equation}
If additionally $\D  f =0 $ then
\begin{equation}\label{eq:inverse_div_for_Delta_u}
  \mathcal{R} \Delta f = \nabla f + (\nabla f)^{T} .
\end{equation}
\end{lemma}

Thanks to Lemma \ref{lemma:inverse_div} we need to estimate the following new Reynolds stress defined by using the inverse divergence operator $\mathcal{R}$.
\begin{equation}
\begin{aligned}\label{eq:newR}
R_{n}  = & \underbrace{\mathcal{R}(\D w_n  \otimes u_{n-1}+u_{n-1} \otimes w_n  )}_{\text{quadratic error}}\\
 &+ \underbrace{\mathcal{R}(\Delta  w_n  ) }_{\text{linear error}}   + \underbrace{\mathcal{R}(\D w_n^p \otimes  w_n^p + \overline{R}_l)}_{\text{oscillation error}} + \underbrace{\mathcal{R}\D ( w_n^c \otimes  w_n^p +  w_n^p \otimes  w_n^c +   w_n^c \otimes  w_n^c)}_{\text{correction error}}
\end{aligned} 
\end{equation}
which is well-defined since all terms involved have zero-mean, and we simply denote the equation as
\begin{equation}\label{eq:newR_short}
R_{n} = E_q +E_l +E_o +E_c .
\end{equation}
It suffices to check that for each part we have
\begin{equation*}
\max \{\|E_q \|_1 , \|E_l \|_1 , \|E_o \|_1 ,\|E_c \|_1 \}  \leq \frac{1}{4}\delta_{n+1}.
\end{equation*}

\subsubsection*{Oscillation error}
Due to the fact that each $\chi_{i,n} W_{[i]}$ has disjoint support we can compute the nonlinear term as
\begin{align*}
w_n^p \otimes w_n^p  &=   \sum_{i} 4^{i+1} \delta_{n}\l_{n-1}^{-2 \alpha }  \chi_{i,n}^2 W_{[i]}  \otimes  W_{[i]}^{\mu_n}\Big(\Id - \frac{ \overline{R}_l }{4^{i+1}\delta_{n}\l_{n-1}^{-2 \alpha }}, \sigma_n x   \Big) .
\end{align*}
From Lemma \ref{lemma:mikado_R} it follows
\begin{align} 
w_n^p \otimes w_n^p  &= -\overline{R}_l+ \sum_{i} 4^{i+1} \delta_{n}  \l_{n-1}^{-  2\alpha }  \chi^2_{i,n}  \Id + \sum_{i,k} \rho_{i,k,n}^2  \phi_{i,k}^{\mu_{n}}(\sigma_{n} \cdot) ~ k\otimes k   \label{eq:osc_interation}
\end{align}
where scaler functions $\rho_{i,k,n}\in C^\infty(\TT^d)$ and $\phi_{i,k}^{\mu_{n}} \in C_0^\infty(\TT^d)$ are defined respectively as
$$
\rho_{i,k,n}= 2^{i+1} \delta_{n}^\frac{1}{2}  \l_{n-1}^{-   \alpha }  \chi_{i,n}   \Gamma_k\Big(\Id-\frac{ \overline{R}_l }{4^{i+1}\delta_{n}\l_{n-1}^{-2 \alpha }}    \Big)
$$ 
and 
$$
\phi_{i,k}^{\mu_{n}} =  \big(\psi_{i,k}^{\mu_n}\big)^2 -\fint  \big(\psi_{i,k}^{\mu_n}\big)^2 .
$$
Upon taking divergence, we can find a pressure $P$ to absorb the second term in \eqref{eq:osc_interation} so that
\begin{align*}
\D(  w_n^p \otimes w_n^p  ) + \nabla P &= -\D \overline{R}_l  + \sum_{i,k}\D  \Big( \rho_{i,k,n}^2 \phi_{i,k}^{\mu_{n}}(\sigma_{n} \cdot) ~ k\otimes k      \Big)  
\end{align*}
Noticing the fact that 
$$ 
\D  \Big(  \phi_{i,k}^{\mu_{n}}(\sigma_{n} x) ~ k\otimes k      \Big)  =0
$$ 
we get
\begin{align*}
E_o   & =  \mathcal{R}(\D w_n^p \otimes  w_n^p + \D \overline{R}_l)  \\
& =  \sum_{i,k} \mathcal{R}\big[\nabla \rho_{i,k,n}^2 \phi_{[i],k}^{\mu_{n}}(\sigma_{n} \cdot) ~ k\otimes k      \big]  .
\end{align*}
For the remainder of this part, we fix some $p>1$ (depending on $d$ and $\alpha$) sufficiently close to $1$ such that $L^{p}(\TT^d) \hookrightarrow W^{-\alpha,1}(\TT^d)$, and will estimate $\| E_o\|_p$.

Since $a_{i,k,n} =\rho_{i,k,n} ~k $, we get from Proposition \ref{prop:coef_a_b} that 
$$
\|\nabla^m \rho_{i,k,n}^2 \|_\infty \lesssim \delta_{n}^\frac{1}{2} \l_{n-1}^{-  \alpha } l^{ -m -2d-2} \quad \text{for all $m\in \NN$}
$$ 
and by definition that
$$
\fint_{\TT^d } \phi_{[i],k}^{\mu_{n}}(\sigma_{n} \cdot) =0.
$$

Hence we have the following estimate for $E_o$ by invoking the second part of Proposition \ref{prop:inversediv_commutator2} with $C_a=\delta_{n}^\frac{1}{2} \l_{n-1}^{-  \alpha } l^{  -2d-2}$, $\mu = l^{-1}$, $\sigma=\sigma_n$ and $p$:
\begin{align}\label{eq:oscE1_inversediv}
\| E_{o  } \|_p
&\lesssim \sum_{i,k} \sigma_{n}^{-1+\alpha} \big\| |\nabla|^{-\alpha} \big(\nabla \rho_{i,k,n}^2 \phi_{[i],k}^{\mu_{n}}(\sigma_{n} \cdot )   \big) \big\|_p + \sigma_n^{-100d}
\end{align}
where we have used the bound
\begin{align}\label{eq:phi_L1_1}
\big\|\phi_{[i],k}^{\mu_{n}}(\sigma_{n} \cdot ) \big\|_1 \lesssim 1.
\end{align}
Then the embedding $L^{p}(\TT^d) \hookrightarrow W^{-\alpha,1}(\TT^d)$ implies that
\begin{align}\label{eq:oscE1_inversediv2}
\| E_{o  } \|_p
&\lesssim \sum_{i,k} \sigma_{n}^{-1+\alpha} \big\|   \nabla \rho_{i,k,n}^2 \phi_{[i],k}^{\mu_{n}}(\sigma_{n} \cdot )    \big\|_1 + \sigma_n^{-100d}
\end{align}
Since it is easy to see that
$$
\sigma_n^{-10d} \ll \delta_{n+1},
$$
to bound $\| E_o\|$, it suffices to bound $ \big\| \nabla \rho_{i,k,n}^2 \phi_{[i],k}^{\mu_{n}}(\sigma_{n} \cdot )   \big\|_1$.
We attempt to apply the first part of Proposition \ref{prop:inversediv_commutator2} with the same parameters, but \eqref{eq:oscE1_inversediv} is in $L^1$ rather than $L^2$ and if one uses the small support argument as in the proof of Proposition \ref{prop:w_p} and \ref{prop:w_c}, one has to bound $\|\nabla \rho_{i,k,n} \|_2$, which will be too big and have no decay in view of Proposition \ref{prop:coef_a_b}. To resolve this issue, we appeal to the following heuristics:
$$
\big\| \nabla \rho_{i,k,n}^2 \phi_{[i],k}^{\mu_{n}}(\sigma_{n} \cdot )   \big\|_1 \lesssim   l^{-1} \big\|   \rho_{i,k,n}  \psi_{[i],k}^{\mu_{n}}(\sigma_{n} \cdot )   \big\|_2^2 +\text{Error terms}
$$
We will show a slightly weaker bound in the following. Firstly, Let 
$$
\rho_l=\mathbb{P}_{\leq l^{-1-\alpha}}\rho_{i,k,n}^2  
$$
where the extra factor $\alpha$ will allow us to exploit the derivative bounds for $a_{i,k,n}$.
More precisely, applying the same integration by parts argument as in the proof of Proposition \ref{prop:inversediv_commutator2}, one can show that
\begin{align}\label{eq:smallhighfre_a_ikn}
\|\nabla (\rho_{i,k,n}^2 -\rho_l) \|_\infty +\| \rho_{i,k,n}^2  -\rho_l  \|_\infty\lesssim l^{1000d}
\end{align}
where the implicit constant depends on $\alpha$ and $d$. Using this and \eqref{eq:phi_L1_1}, we have
\begin{align}\label{eq:E1_a2_to_rho_l}
\big\| \nabla \rho_{i,k,n}^2 \phi_{[i],k}^{\mu_{n}}(\sigma_{n} \cdot )   \big\|_1 \lesssim \big\| \nabla \rho_l \phi_{[i],k}^{\mu_{n}}(\sigma_{n} \cdot )   \big\|_1+l^{1000d}.
\end{align}
Thus it suffices to get rid of the derivative on $\rho_l$ and then bound $\big\| \rho_l \phi_{[i],k}^{\mu_{n}}(\sigma_{n} \cdot )   \big\|_1$ .
Using the convolution representation of $ \rho_l$ by the Littlewood-Paley theory 
\begin{align*}
\rho_l= \mathbb{P}_{\leq 2l^{-1-\alpha}}\rho_l 
\end{align*}
we have
\begin{align*}
\big\| \nabla \rho_l \phi_{[i],k}^{\mu_{n}}(\sigma_{n} \cdot )   \big\|_1 =\int \big|\phi_{[i],k}^{\mu_{n}}(\sigma_{n} x )\big|\bigg|\int \rho_l(x-y)\nabla \widetilde{\varphi}_{l^{-1-\alpha}}(y )  dy  \bigg|  dx
\end{align*}
where $ \widetilde{\varphi}_{l^{-1-\alpha}}$ is the Fourier inverse for the frequency cut-off $\mathbb{P}_{\leq 2l^{-1-\alpha}}$.
So by the Fubini's theorem and the bound 
$$
\big\|\nabla \widetilde{\varphi}_{l^{-1-\alpha}} \big\|_1  \lesssim l^{-1-\alpha}
$$
we find that
\begin{align}\label{eq:sup_y_rho_l_L1}
\big\| \nabla \rho_l \phi_{[i],k}^{\mu_{n}}(\sigma_{n} \cdot )   \big\|_1 &=\int \int \Big|\rho_l(x-y)   \phi_{[i],k}^{\mu_{n}}(\sigma_{n} x )\Big| dx \big|\nabla \widetilde{\varphi}_{l^{-1-\alpha}}(y ) \big|  dy  \nonumber \\
& \leq \big\|\nabla \widetilde{\varphi}_{l^{-1-\alpha}} \big\|_1 \sup_y \big\|\rho_l(\cdot - y) \phi_{[i],k}^{\mu_{n}}(\sigma_{n} \cdot ) \big\|_1\nonumber \\
&\lesssim l^{-1-\alpha} \sup_y \big\|\rho_l(\cdot - y) \phi_{[i],k}^{\mu_{n}}(\sigma_{n} \cdot ) \big\|_1.
\end{align}
Thanks to \eqref{eq:smallhighfre_a_ikn} and \eqref{eq:phi_L1_1} again, we get
$$
\Big\| \big[\rho_{i,k,n}^2(\cdot - y)  - \rho_l(\cdot - y)    \big] \phi_{[i],k}^{\mu_{n}}(\sigma_{n} \cdot ) \Big\|_1 \lesssim \big\|   \rho_{i,k,n}^2   - \rho_l        \big\|_\infty  \big\| \phi_{[i],k}^{\mu_{n}}(\sigma_{n} \cdot ) \big\|_1 \lesssim l^{1000d}
$$
where the implicit constant is independent of $y$. Then \eqref{eq:sup_y_rho_l_L1} becomes
\begin{align}\label{eq:sup_y_a_ikn_L1}
\big\| \nabla \rho_l \phi_{[i],k}^{\mu_{n}}(\sigma_{n} \cdot )   \big\|_1 &\lesssim l^{-1-\alpha} \sup_y \big\|\rho_{i,k,n}^2(\cdot - y) \phi_{[i],k}^{\mu_{n}}(\sigma_{n} \cdot ) \big\|_1+l^{1000d}.
\end{align}
Putting together \eqref{eq:oscE1_inversediv}, \eqref{eq:E1_a2_to_rho_l}, and \eqref{eq:sup_y_a_ikn_L1} and using the fact that $\sigma_{n} \leq l^{-10}$ we have
\begin{align}\label{eq:Eo_no_nabla}
\| E_{o  } \|_p
&\lesssim \sum_{i,k} \sigma_{n}^{-1+\alpha} l^{-1-\alpha} \sup_y \big\|\rho_{i,k,n}^2(\cdot - y) \phi_{[i],k}^{\mu_{n}}(\sigma_{n} \cdot ) \big\|_1 + \sigma_n^{-10d} .
\end{align}
For each fix $y \in \TT^d$ we compute that
\begin{align}\label{eq:Eo_a_psi_L1toL2}
\big\|\rho_{i,k,n}^2(\cdot - y) \phi_{[i],k}^{\mu_{n}}(\sigma_{n} \cdot ) \big\|_1 \leq \big\|\rho_{i,k,n} (\cdot - y)  \psi_{[i],k}^{\mu_{n}}(\sigma_{n} \cdot )   \big\|_2^2 + \big\|\rho_{i,k,n} \big \|_2^2 
\end{align}
and now we can apply the first part of Proposition \ref{prop:inversediv_commutator2} with the parameters $C_a=\delta_{n}^\frac{1}{2} \l_{n-1}^{-  \alpha } l^{  -2d-2}$, $\mu = l^{-1}$, and $\sigma=\sigma_n$ to obtain
\begin{align}\label{eq:Eo_a_psi_L2}
 \big\|\rho_{i,k,n} (\cdot - y)  \psi_{[i],k}^{\mu_{n}}(\sigma_{n} \cdot )   \big\|_2  & \lesssim \big\|\rho_{i,k,n} (\cdot - y)\big\|_2  \big\|\psi_{[i],k}^{\mu_{n}}(\sigma_{n} \cdot )   \big\|_2  +\sigma_n^{-10d} \nonumber\\
&\lesssim  \delta_{n }^{\frac{1}{2}} \l_{n-1 }^{- \alpha}     +\sigma_n^{-10d} 
\end{align}
where we have used Proposition \ref{prop:coef_a_b} to get the bound of $\|\rho_{i,k,n} \|_2\sim \|a_{i,k,n} \|_2 $.
Therefore, from  \eqref{eq:Eo_no_nabla}, \eqref{eq:Eo_a_psi_L1toL2}, and \eqref{eq:Eo_a_psi_L2} it follows that
\begin{align*}
\| E_{o  } \|_p
&\lesssim \sum_{i,k} \sigma_{n}^{-1+\alpha} l^{-1-\alpha}\delta_{n }  \l_{n-1 }^{-2 \alpha}   .
\end{align*}
Again by taking a sufficiently large $a$ and using $\l_{n-1 }^{-2 \alpha}$ to absorb the constant and the logarithmic factor causing by the summation in $i$, we can ensure that
\begin{align*}
\| E_{o  } \|_1 &\leq \| E_{o  } \|_p \leq  \frac{1}{4}\sigma_n^{-1+\alpha} l^{-1-\alpha} \delta_{n }   .
\end{align*}
In view of the choice of constants $d \geq 4$, $\alpha = 10^{-6}$, $\beta =\frac{1}{200}$ and $b=5$, we have the following numerical inequality
\begin{align*}
- \frac{1+\alpha}{4}+ \frac{(1-\beta+\alpha)(1+\alpha)}{b} -\beta    < -2b \beta
\end{align*}
which implies that
\begin{align}\label{eq:final_E_o}
\| E_{o  } \|_1 \leq \frac{1}{4} \sigma_n^{-1} l^{-1-\alpha} \delta_{n }   \leq  \frac{1}{4}   \delta_{n+1 }   .
\end{align}

\subsubsection*{Linear error}
For the linear error, we first use Lemma \ref{lemma:inverse_div} to obtain
\begin{align*}
\|E_l \|_1 & =  \| \mathcal{R}\Delta  w_n \|_1 \leq  2 \|   \nabla w_n  \|_1 .
\end{align*}
Then we can simply use the estimates \eqref{eq:w_p_L1} and \eqref{eq:w_c_L1} to get
\begin{align*} 
\|E_l \|_1 & \leq    2\| \nabla w_n^p \|_1 +2\| \nabla w_n^c \|_1 \\
& \leq  \frac{1}{4}\delta_{n}^\frac{1}{2}    \l_{n} \mu_n^{\frac{-d+1}{2}}   . 
\end{align*}
To check the validity of $\|E_l \|_1 \leq \frac{1}{4} \delta_{n+1}$, we need to make sure that
\begin{align}
\delta_{n}^\frac{1}{2}    \l_{n} \mu_n^{\frac{-d+1}{2}}  \leq \delta_{n+1}
\end{align}
which after taking logarithm and using the definitions of various constants is equivalent to 
\begin{align}
-\beta + 1 + \frac{3}{4}\frac{1-d}{2} \leq -2b\beta.
\end{align}
Since $d \geq 4$, $\beta = \frac{1}{200}$ and $b=5$ the above inequality holds trivially. So we can conclude that
\begin{align}\label{eq:final_E_l}
\|E_l \|_1 \leq \frac{1}{4} \delta_{n+1}.
\end{align}

\subsubsection*{Quadratic error}
Thanks to Lemma \ref{lemma:inverse_div}, we need to estimate the terms
\begin{align*}
\|E_q \|_1 & \leq    \|\mathcal{R} \D  (\ul \otimes w_n^p) \|_{1}  +\|\mathcal{R} \D  (\ul \otimes w_n^p) \|_{1} \\
&:= \|E_{q1} \|_1 +\|E_{q2} \|_1 .
\end{align*}
Let us show that $\|E_{qj} \|\leq \frac{1}{8} \delta_{n+1} $ for $j=1,2$ in the following.

For the first term $E_{q1}$, we have by the $L^p$ boundedness of the Riesz transform, $p>1$ that
\begin{align*}
\|E_{q1} \|_1 \leq \|E_{q1} \|_p \lesssim_p \sum_{i,k} \|\ul \otimes  a_{i,k,n}  \psi_{i,k}^{\mu_n}(\sigma_n \cdot) \|_p.
\end{align*}
Then by H\"older's inequality and the fact that $\big| \Supp \varphi_{i,k,n}\big| \lesssim \mu_n^{-d+1}$
\begin{align}\label{eq:Qua_holder}
\|\ul \otimes  a_{i,k,n}  \psi_{i,k}^{\mu_n}(\sigma_n \cdot) \|_p &\lesssim \| \ul\|_\infty   \|   a_{i,k,n}  \varphi_{i,k,n} \|_2 \mu_n^{(d-1) ( \frac{1}{p} -\frac{1}{2})} 
\end{align}
From \eqref{eq:w_p_a_ik_L2} in Proposition \ref{prop:w_p} we know that
\begin{align}\label{eq:Qua_w_p_a_ik_L2}
\|   a_{i,k,n}  \varphi_{i,k,n} \|_2 \lesssim \delta_n^\frac{1}{2} \l_{n-1 }^{-\alpha}
\end{align}
and from Lemma \ref{lemma:mollidied_u_estimate} and the Sobolev embedding  $H^{\frac{d+1}{2}} \hookrightarrow L^\infty$ we get
\begin{align}\label{eq:Qua_ul_infty}
\|\ul \|_\infty \lesssim \|\ul \|_{H^{\frac{d+1}{2}} } \lesssim l^{- \frac{d+1}{2}}
\end{align}
Now choosing $p>1$ sufficiently close to $1$ such that
$$
\mu_n^{(d-1) ( \frac{1}{p} -\frac{1}{2})} \leq \mu_n^{    -\frac{d-1 }{2} } \l_{n-1}^\alpha
$$
and combining \eqref{eq:Qua_holder}, \eqref{eq:Qua_w_p_a_ik_L2} and \eqref{eq:Qua_ul_infty} we have
$$
\|\ul \otimes  a_{i,k,n}  \psi_{i,k}^{\mu_n}(\sigma_n \cdot) \|_p  
 \lesssim l^{- \frac{d+1}{2}}   \delta_n^\frac{1}{2}  \mu_n^{    -\frac{d-1 }{2} } 
$$

Since $d \geq 4$, $b=5$ and $\beta =\frac{1}{200}$, by taking a sufficiently large $a$ we have
\begin{equation*}
l^{- \frac{d+1}{2}}   \delta_n^\frac{1}{2}  \mu_n^{    -\frac{d-1 }{2} }    \ll\l_{n+1}^{-2\beta} =\delta_{n+1 }.
\end{equation*}
So it follows that
\begin{align}\label{eq:final_E_q1}
\|E_{q1} \|_1 \leq \frac{1}{8} \delta_{n+1 }  .
\end{align}

For $E_{q2}$ we will use a simple argument to get a very crude bound that suffices for our purpose. We first obtain by the $L^p$ boundedness of the Riesz transform, $p>1$ and H\"older's inequality that 
\begin{align*}
\|E_{q2} \|_1 \leq \|E_{q2} \|_p \lesssim_p \|\ul \otimes  w_n^c\|_p \leq \|\ul\|_2    \| w_n^c\|_{\frac{2p}{2-p}}.
\end{align*}
Then choosing $p>1$ sufficiently close to $1$ such that in view of Proposition \ref{prop:w_c} we have
$$
\| w_n^c\|_{\frac{2p}{2-p}} \leq \frac{1}{8} \delta_{n }^\frac{1}{2} \mu_n^{\frac{-d+1}{2}} l^{-1} \sigma_n^{-1} \l_{n-1}^\alpha
$$
Then we get
\begin{align*}
\|E_{q2} \|_1 \lesssim   \delta_{n}^\frac{1}{2}    l^{-1}  \sigma_{n}^{-1}i_{\max} \l_{n-1}^\alpha.
\end{align*}
Again using $d \geq 4$, $b=5$ and $\beta =\frac{1}{200}$ we find that
\begin{align*}
\frac{1-\beta+2\alpha}{b} - \frac{1}{4}  < -2b \beta.
\end{align*}
Since for any $\epsilon>0$, there exists $a$ sufficiently large so that $i_{\max} \leq \l_n^{\epsilon}$. Then by taking $a$ sufficiently large we can ensure that
\begin{align*}
\|E_{q2} \|_1 \lesssim   \delta_{n}^\frac{1}{2}    l^{-1}  \sigma_{n}^{-1} i_{\max} =   \l_n^{ \frac{1-\beta+2\alpha}{b} - \frac{1}{4}} i_{\max} \ll \delta_{n+1 },
\end{align*}
So provided that $a$ is large enough, we can conclude that
\begin{align}\label{eq:final_E_q2}
\|E_{q2} \|_1 \leq \frac{1}{8} \delta_{n+1 }.
\end{align}

\begin{remark}
One may notice that the bound we obtained for $E_{q2}$ is worse than that of $E_{q1} $. In fact $E_{q2}$ should be much smaller than $E_{q1}$ since $w^c_n$ is much smaller than $w^p_n$. As we do not plan to obtain the optimal regularity of the final solution $u$, a rougher bound for $E_{q2}$ still suffices for our purpose.
\end{remark}

\subsubsection*{Correction error}
It follows directly from Lemma \ref{lemma:inverse_div} and H\"older's inequality that
\begin{align*}
\|E_c \|_1 & = \| \mathcal{R}\D ( w_n^c \otimes  w_n^p +  w_n^p \otimes  w_n^c +   w_n^c \otimes  w_n^c) \|_1 \\
& \leq\|  w_n^c \otimes  w_n^p +  w_n^p \otimes  w_n^c +   w_n^c \otimes  w_n^c     \|_1 \\
& \leq  \|  w_n^c\|_2  \| w_n^p\|_2  +  \|w_n^c \|_2^2   .
\end{align*}
From \eqref{eq:w_p_L2} and \eqref{eq:w_c_L2} we get 
\begin{align*}
\|w_n^c \|_2^2 &  \leq  \frac{1}{16}\delta_{n}   \l_{n}^{-2} l^{-2}  \\
\|  w_n^c\|_2  \| w_n^p\|_2  &  \leq  \frac{1}{16} \delta_{n}    \l_{n}^{-1} l^{-1} .
\end{align*}
So 
\begin{align*}\label{eq:new_R_E_c}
\|E_c \|_1  &  \leq \frac{1}{4} \delta_{n}    \l_{n}^{-1} l^{-1}
\end{align*}
which is smaller than the final estimate for $E_{q2}$ and thus we conclude that
\begin{equation}\label{eq:final_E_c}
\|E_c \|_1     \leq \frac{1}{4} \delta_{n+1}   .
\end{equation}

\stepcounter{steps}
\subsection*{Step \arabic{steps}: Check inductive hypothesis \texorpdfstring{\eqref{eq:hypothesis1}}{H1}}
Adding up estimates \eqref{eq:final_E_o}, \eqref{eq:final_E_l}, \eqref{eq:final_E_q1}, \eqref{eq:final_E_q2} ,and \eqref{eq:final_E_c} we obtain the bound for the new Reynolds stress
\begin{equation}
\| R_{n} \|_1 \leq \|E_o \|_1+\|E_l \|_1+\|E_{q1} \|_1+\|E_{q2} \|_1+\|E_c \|_1 \leq \delta_{n+1}   .
\end{equation}
So \eqref{eq:hypothesis1} is verified and the proof of Proposition \ref{prop:main_iteration} is completed.

\appendix

\section{An Estimate for H\"older norms}
We collect here the following classical result on the H\"older norms of composition of functions. A proof using the multivariable chain rule can be found in \cite{MR3254331}.
\begin{proposition}\label{prop:holder_composition}
Let $F : \Omega \to \RR$ be a smooth function with $\Omega \subset \RR^d$. For any smooth function $u :\RR^d \to \Omega$ and any $1 \leq m \in \NN$ we have
\begin{align}\label{eq:holder_composition}
\|\nabla^m(F\circ u) \|_\infty &\lesssim \|\nabla^{m}   u \|_\infty \sum_{1\leq i\leq m} \|\nabla^i F \|_\infty  \| u \|_{\infty}^{i-1}  
\end{align}
where the implicit constant depends on $m$, $d$.
\end{proposition}

\section{Constantin-E-Titi commutator estimate}\label{appendix:CET}
The following commutator-type estimate originates from the one proved in \cite{MR1298949}. Compared with other versions used in \cite{1701.08678,MR3374958} the one stated below is homogeneous, i.e. it only involves highest order derivative. For reader's continence we include a proof here following closely the argument of Lemma 1 in \cite{MR3289360}.
\begin{proposition}\label{prop:CET_commutator}
Let $f,g \in C^\infty(\TT^d)$ and let $\eta_\epsilon$ be a family of mollifier. For any $m\in \NN$ and $1\leq p \leq \infty$ we have
\begin{equation}\label{eq:CET_commutator}
\Big\|\nabla^m \big[(fg)*\eta_\epsilon -(f*\eta_\epsilon )(g*\eta_\epsilon)  \big] \Big\|_p\lesssim  \epsilon^{ 2-m} \|\nabla  f\|_{2p}  \| \nabla g \|_{2p}
\end{equation}
\end{proposition}
\begin{proof}
It suffices to prove for any multi-index $\alpha$ with $|\alpha|=m$ the following estimate:
\begin{equation}
\Big\|\partial_\alpha \big[(fg)*\eta_\epsilon -(f*\eta_\epsilon )(g*\eta_\epsilon)  \big] \Big\|_p\lesssim  \epsilon^{ 2-m} \|\nabla  f\|_{2p}  \| \nabla g \|_{2p}.
\end{equation}
By the product rule and the fact that mollification commutes with differentiation we compute
\begin{align*}
\partial_\alpha & \big[   (fg)*\eta_\epsilon  - (f*\eta_\epsilon )(g*\eta_\epsilon) \big]\\
&=  (fg)* \partial_\alpha \eta_\epsilon -  \sum_{\beta } C_\beta^\alpha (f* \partial_{\alpha-\beta} \eta_\epsilon )(g*\partial_{\beta}\eta_\epsilon) 
\end{align*}
where the summation is taking over all multi-index $0 \leq \beta \leq \alpha$. So
\begin{align*}
\partial_\alpha  \big[  (fg)*\eta_\epsilon - (f*\eta_\epsilon )(g*\eta_\epsilon)\big] 
&=   (fg)* \partial_\alpha \eta_\epsilon - (f* \partial_{\alpha } \eta_\epsilon )(g* \eta_\epsilon)- (f*  \eta_\epsilon )(g* \partial_{\alpha } \eta_\epsilon)  \\
 & \qquad- \sum_{\beta\neq 0,\alpha }  C_\beta^\alpha (f* \partial_{\alpha-\beta} \eta_\epsilon )(g*\partial_{\beta}\eta_\epsilon)
\end{align*}

The fact that $\int \eta_\epsilon =1$ and $\int \partial \eta_\epsilon =0$ implies
\begin{align*}
f*  \eta_\epsilon  =   [f-f(x)]*  \eta_\epsilon+ f(x) \quad \text{and }\quad f*  \partial_\beta \eta_\epsilon   = [f-f(x)]*  \partial_\beta\eta_\epsilon
\end{align*}
for any multi-index $\beta\neq 0$.
Let 
$$
r_\alpha(f,g)=\int \big[f(x-y) -f(x)\big]     \big[g(x-y) -g(x)\big]  \partial_\alpha \eta_\epsilon(y)dy,
$$
and it follows
\begin{align*}
r_\alpha(f,g) 
&= (fg)*  \partial_\alpha \eta_\epsilon    -(f*  \eta_\epsilon)( g *  \partial_\alpha \eta_\epsilon  ) -  ( f *  \partial_\alpha \eta_\epsilon ) (g*  \eta_\epsilon)\\
& \qquad + [f-f(x)]*  \eta_\epsilon( g *  \partial_\alpha \eta_\epsilon  ) +  ( f *  \partial_\alpha \eta_\epsilon ) [g-g(x)]*  \eta_\epsilon\\
& =(fg)*  \partial_\alpha \eta_\epsilon    -(f*  \eta_\epsilon)( g *  \partial_\alpha \eta_\epsilon  ) -  ( f *  \partial_\alpha \eta_\epsilon ) (g*  \eta_\epsilon)\\
& \qquad + [f-f(x)]*  \eta_\epsilon( [g-g(x)] *  \partial_\alpha \eta_\epsilon  ) +  ( [f-f(x)] *  \partial_\alpha \eta_\epsilon ) [g-g(x)]*  \eta_\epsilon
\end{align*}
and
\begin{align*}
\sum_{\beta\neq 0,\alpha }  C_\beta^\alpha (f* \partial_{\alpha-\beta} \eta_\epsilon )(g*\partial_{\beta}\eta_\epsilon) = \sum_{\beta\neq 0,\alpha }  C_\beta^\alpha \Big[(f-f(x))* \partial_{\alpha-\beta} \eta_\epsilon \Big]    \Big[(g-g(x))*\partial_{\beta}\eta_\epsilon  \Big].
\end{align*}
Putting together the preceding two equations we have
\begin{align*}
\partial_\alpha  \big[  (fg)*\eta_\epsilon - (f*\eta_\epsilon )(g*\eta_\epsilon)\big] 
&=   r_\alpha(f,g) - \sum_{\beta }  C_\beta^\alpha \big(   f-f(x)  \big)*\partial_{\alpha-\beta}\eta_\epsilon \cdot \big( g-g(x)   \big)  *\partial_{\beta}\eta_\epsilon.
\end{align*}

On the one hand by Minkowski's inequality we have
\begin{align*}
\Big\|  \int \big[f(x-y) -f(x)\big]  &  \big[g(x-y) -g(x)\big]  \partial_\alpha \eta_\epsilon(y)dy\Big\|_p \lesssim \\
& \qquad    \int \big\|  f(\cdot-y) -f(\cdot)\big\|_{2p}  \big\| g(\cdot-y) -g(\cdot) \big\|_{2p}  \partial_\alpha \eta_\epsilon(y)dy .
\end{align*}
From the integral form of Mean Value Theorem and Minkowski's inequality it follows that
\begin{align*}
& \big\|  f(\cdot-y) -f(\cdot)\big\|_{2p} \lesssim |y| \| \nabla f \|_{2p} \\
&\big\|  g(\cdot-y) -g(\cdot)\big\|_{2p} \lesssim |y| \| \nabla g \|_{2p} 
\end{align*}
which enables us to obtain
\begin{align*}
\Big\|  \int \big[f(x-y) -f(x)\big]   \big[g(x-y) -g(x)\big]  \partial_\alpha \eta_\epsilon(y)dy  \Big\|_p  &\lesssim 
   \| \nabla f \|_{2p}  \| \nabla g \|_{2p}  \int  |y|^2 \partial_\alpha \eta_\epsilon(y)dy \\
&\lesssim  \epsilon^{ 2-m} \|\nabla  f\|_{2p}  \| \nabla g \|_{2p}.
\end{align*}

On the other hand by H\"older's inequality we have
\begin{align*}
\Big\| \sum_{\beta }   C_\beta^\alpha \Big[ (f-f(x))* \partial_{\alpha-\beta} \eta_\epsilon    \Big] &\Big[(g-g(x))*\partial_{\beta}\eta_\epsilon\Big] \Big\|_p \\
& \lesssim \sum_{\beta }  \big\|   (f-f(x))* \partial_{\alpha-\beta} \eta_\epsilon \big\|_{2p}   \big\| (g-g(x))*\partial_{\beta}\eta_\epsilon)   \big\|_{2p} \\
& \lesssim \epsilon^{m-2} \|  \nabla f \|_{2p}\|  \nabla g \|_{2p}
\end{align*}
where we have used the fact that $\| (f-f(x))* \partial_{ \beta} \eta_\epsilon\|_{2p} \lesssim \epsilon^{\beta -1 } \| \nabla f \|_{2p}  $.

Therefore
\begin{equation*}
\Big\|\partial_\alpha \big[(fg)*\eta_\epsilon -(f*\eta_\epsilon )(g*\eta_\epsilon)  \big] \Big\|_p\lesssim  \epsilon^{ 2-m} \|\nabla  f\|_{2p}  \| \nabla g \|_{2p}.
\end{equation*}
\end{proof}

\section{Proof of Proposition \ref{prop:inversediv_commutator2}}\label{appendix:proofofproposition}
We provide a proof of Proposition \ref{prop:inversediv_commutator2} using the Littlewood-Paley decomposition.

\begin{proof}[Proof of Proposition \ref{prop:inversediv_commutator2}]
 
By considering $\widetilde{a}:= \frac{1}{C_a} a$ it suffices to prove both of the results for $C_a =1$. Notice that since $p\geq 2$ is even, the function $a^p$ as a composition of $a:\TT^d \to [-1,1]$ and $x^p$ is smooth. Therefore, applying Proposition \ref{prop:holder_composition} we see that
\begin{equation}\label{eq:removedabs}
\begin{aligned}
\| \nabla^m |a|^p \|_\infty &\lesssim_p   \|\nabla^m a \|_\infty + \sum_{i\leq m} \|\nabla a \|_\infty^{i-1}\nonumber \\
&\lesssim_p  \mu^{m} 
\end{aligned}\qquad \text{for any $m\in \NN$}.
\end{equation}

We can now introduce the split:
\begin{align*}
\|    a     f \|_p^p 
& \leq \Big| \int_{\TT^d}  (a^p -\overline{|   a |^p }  )  (|    f|^p -\overline{|    f|^p } )  dx    \Big|+  2\|   a \|_p^p \|f \|_p^p .
\end{align*}
We will apply a standard integration by parts argument to get\footnote{The nonlocal operators $|\nabla|^s $ and $|\nabla|^{-s}$ are defined respectively by multipliers with symbols $|k|^{s}$ and $|k|^{-s}$ for $k\neq 0$.}
\begin{align*}
\|    a     f \|_p^p  
& \leq \Big| \int_{\TT^d} |\nabla|^{M}  (a^p -\overline{|   a |^p }  ) |  \nabla|^{-M}(|    f|^p -\overline{|    f|^p } )  dx    \Big|+  2\|   a \|_p^p \|f \|_p^p  .
\end{align*}
We need show the first term is very small. By H\"older's inequality:
\begin{align}\label{eq:smallerror}
\Big| \int_{\TT^d} |\nabla|^{M}     (a^p -\overline{|   a |^p }  ) |  \nabla|^{-M}(|    f|^p -\overline{|    f|^p } )  dx    \Big| \lesssim \| |\nabla|^{M}      a^p\|_2 \Big\| |  \nabla|^{-M}(|    f|^p -\overline{|    f|^p } ) \Big\|_2 
\end{align}
 
By the $L^2$ boundedness of Riesz transform we can replace the nonlocal $|\nabla|^{M}$ by $\nabla^M$ to obtain
\begin{align}\label{eq:removednonlocal}
\Big\| |\nabla|^{M}      a^p \Big\|_2 \lesssim  \Big\| \nabla^{M}       a^p  \Big\|_2.
\end{align}
Since the domain is $\TT^d$, due to the estimate \eqref{eq:removedabs} we see that
\begin{align}\label{eq:removedsobolev}
 \Big\| \nabla^{M}       a^p  \Big\|_2  \leq  \big\| \nabla^{M}       a^p  \big\|_{\infty} \lesssim \mu^{M }.
\end{align}
Thus, putting together \eqref{eq:removednonlocal}, \eqref{eq:removedsobolev} we get
\begin{align}\label{eq:smallerror1}
\Big\| |\nabla|^{M}       a^p  \Big\|_{2}   \lesssim \mu^{M }.
\end{align}

We turn to estimate the second factor in \eqref{eq:smallerror}. Considering the fact that the function $(|    f|^p -\overline{|    f|^p } ) $ is zero-mean and $\sigma^{-1} \TT^d $-periodic we have
\begin{align*}
\Big\| |  \nabla|^{-M}(|    f|^p -\overline{|    f|^p } ) \Big\|_2  &\lesssim  \sigma^{-M+d} \Big\| \nabla|^{-d} (|    f|^p -\overline{|    f|^p } )   \Big\|_2 \\
&\lesssim \sigma^{-M+d} \Big\|  (|    f|^p -\overline{|    f|^p } )   \Big\|_1\\
& \lesssim \sigma^{-M+d}\|  f \|_p^p
\end{align*}
where the first inequality is a direct consequence of Littlewood-Paley theory and the second inequality follows from the Sobolev embedding $L^1(\TT^d)\hookrightarrow H^{-d}(\TT^d)  $.

So with this and the estimate \eqref{eq:smallerror1} we find that
\begin{align*}
\Big| \int_{\TT^d} |\nabla|^{M}     (a^p -\overline{|   a |^p }  ) |  \nabla|^{-M }(|    f|^p -\overline{|    f|^p } )  dx 
&\lesssim \sigma^{-M+d} \mu^{M }       \|  f  \|_p^p .
\end{align*}
By the assumption that $\mu\leq \sigma^{1-\theta}$, there exists a number $M_{\theta,p,N}\in \NN$ sufficiently large so that 
\begin{equation}
\sigma^{-M+d} \mu^{M} \leq \sigma^{-N p }.
\end{equation}
Then, we have
\begin{align*}
\Big| \int_{\TT^d} |\nabla|^{M}     (a^p -\overline{|   a |^p }  ) |  \nabla|^{-M}(|    f|^p -\overline{|    f|^p } )  dx  
&\lesssim \sigma^{-N p }       \|  f  \|_p^p 
\end{align*}
which finishes the proof of \eqref{eq:holder_commutator1} due to the elementary inequality $(a^p+b^p) \leq (a +b)^p$.

To prove \eqref{eq:inversediv_commutator2} let us first recall the wavenumber projection. For any $\l\in\NN$ define $\mathbb{P}_{\leq \l} =\sum_{q:2^q \leq \l} \Delta_q$ and $\mathbb{P}_{\geq \l} =\Id-\mathbb{P}_{\leq \l}$. 
Consider the following decomposition:
\begin{align*}
|\nabla|^{-1} (a  f)  &=    |\nabla|^{-1+s}   |\nabla|^{-s} \big(\mathbb{P}_{\leq  {2}^{-4}\sigma }  a \big)   f   +   |\nabla|^{-1+s}  |\nabla|^{-s}\big(\mathbb{P}_{\geq {2}^{-4} \sigma }   a  \big) f    \\
&:=|\nabla|^{-1+s} A_{1 }+|\nabla|^{-1+s} A_{2 } 
\end{align*}
 
For the term $A_1$, since $f$ is $\sigma^{-1} \TT^d$-periodic and zero-mean, it follows that
\begin{align*}
\mathbb{P}_{\geq {2}^{-1} \sigma } f  =   f
\end{align*}
and then by the support of Fourier modes of $\big(\mathbb{P}_{\leq {2}^{-4}\sigma }  a \big)   f  $ we have
\begin{align*}
\mathbb{P}_{\leq {2}^{-2}\sigma } \Big[ \mathbb{P}_{\leq {2}^{-4}\sigma } a     f \Big] =0 \quad \text{and} \quad \fint_{\TT^d} \mathbb{P}_{\leq {2}^{-4}\sigma } a     f=0
\end{align*}
which implies that
$$
|\nabla|^{-1+s}  A_1 =  |\nabla|^{-s}  \mathbb{P}_{\geq {2}^{-2}\sigma }A_1
$$

By the Littlewood-Paley theory, we have
$$
\Big\||\nabla|^{-1+s} \mathbb{P}_{\geq {2}^{-2}\sigma }  \Big\|_{L^p \to L^p} \lesssim_p \sigma^{-1} \quad\text{for all $1 < p< \infty$}
$$

So, we have
\begin{align*}
\big\| |\nabla|^{-1+s} A_1   \big\|_p
& \lesssim_p \sigma^{-1} \Big\|  |\nabla|^{-s} \big(\mathbb{P}_{\leq {2}^{-4} \sigma } a     f   \big)\Big\|_p  .
\end{align*}
To get the exact form of the estimate, since and $ |\nabla|^{-1+s} $ is bounded $L^p \to L^p$ for any $1< p < \infty$ we bound the above in the following way:
\begin{align}\label{eq:inversenabla_A1}
\big\|  |\nabla|^{-1+s} A_1   \big\|_p  &  \leq \sigma^{-1+s} \big\|  |\nabla|^{-s}   (a     f )  \big\|_p + \sigma^{-1+s} \Big\|  |\nabla|^{-s} \big(\mathbb{P}_{\geq {2}^{-4} \sigma } a     f   \big)\Big\|_p \nonumber\\
&\lesssim \sigma^{-1+s} \|    |\nabla|^{-s}   (a     f )  \|_p+ \sigma^{-1+s} \big\| \mathbb{P}_{\geq {2}^{-4} \sigma }  a  \big\|_\infty  \|   f   \|_p
\end{align}
Also, for $A_2$ by the same reason, we have
\begin{align*}
\|  |\nabla|^{-1+s}  A_{2 } \|_p  \lesssim    \Big\|\mathbb{P}_{\geq {2}^{-4}\sigma }    a     f \Big\|_p & \leq    \| \mathbb{P}_{\geq {2}^{-4}\sigma }   a \|_\infty  \| f \|_p .
\end{align*}
So it suffices to show $\| \Delta_{q}    a \|_\infty  \lesssim   2^{-N q}$ for all $2^q \geq {2}^{-4} \sigma $. Recall from the definition of the periodic Littlewood-Paley projection  that
\begin{align*}
\Delta_{q}a = \int_{Q^d} \varphi_q(x -y ) a(y ) dy.
\end{align*}
Applying a standard integration by parts argument gives:
\begin{align*}
\Delta_{q}a \leq \Big|  \int_{Q^d}  |\nabla_y|^{-M}\varphi_q(x -y ) |\nabla|^{M}a(y ) dy  \Big|.
\end{align*}
And the below estimate follows from Young's inequality:
\begin{align*}
\| \Delta_{q} a \|_\infty &     \leq    \||\nabla|^{-M}\varphi_q \|_2 \| |\nabla|^{M} a  \|_2  .
\end{align*}
From $L^2$ boundedness of Riesz transform and the assumption on $a$ it follows
\begin{align}\label{eq:nablaMa}
 \| |\nabla|^{M} a  \|_2  \lesssim \| \nabla^{M} a  \|_2 \lesssim \| \nabla^{M} a  \|_\infty \leq \mu^{M} 
\end{align} 
where we have used that $C_a=1$.
By the Littlewood-Paley frequency cutoff there holds the bound
\begin{align}\label{eq:nabla_minusM}
\| |\nabla|^{-M}\varphi_q \|_2 \lesssim 2^{-q M } \| \varphi_q \|_2 \lesssim 2^{-q M +qd }.
\end{align} 
Thus, combining estimates \eqref{eq:nablaMa} and \eqref{eq:nabla_minusM} we find
\begin{align}\label{eq:smalldelta_q}
\| \Delta_{q} a \|_\infty \lesssim  2^{qd} \mu^{M } 2^{-q M }.
\end{align}
Since $\mu \leq \sigma^{1-\theta}$, there exists a sufficiently large $\l_0 \in \NN$ depending on $\theta>0$ so that $\l_0 \leq \mu \ll {2}^{-4} \sigma  $. Then there exists a sufficiently large $M\in \NN$ so that in view of \eqref{eq:smalldelta_q} we have
$$
\| \Delta_{q} a \|_\infty \lesssim   2^{-N   q } \quad \text{for all $ 2^q \geq {2}^{-4} \sigma $}.
$$
After taking a summation in $q$ for $2^q \geq {2}^{-4} \sigma$ we have that
\begin{align*} 
 \| \mathbb{P}_{\geq {2}^{-4}\sigma }   a \|_\infty \lesssim \sigma^{-N  }.
\end{align*}
Then collecting all the estimates, we have
\begin{align*} \big\|   |\nabla|^{-1+s} (a  f) \big\|_p &\leq \big\|   |\nabla|^{-1+s} A_{1 } \big\|_p + \big\| |\nabla|^{-1+s} A_{2 } \big\|_p\\
&\lesssim \sigma^{-1+s}  \|   |\nabla|^{- s} (a  f)  \|_p + \sigma^{-N} \|f \|_p .
\end{align*}

\end{proof}


\end{document}